\documentclass{amsart}
\usepackage[T1]{fontenc}

\usepackage{amsmath}
\usepackage{amssymb}
\usepackage{bm}
\usepackage{mathrsfs}
\usepackage{accents}

\usepackage{float}
\usepackage{graphicx,color}
\usepackage{ascmac}

\usepackage{tabularx}

\usepackage{tikz}
\usetikzlibrary{calc,decorations.markings}

\usepackage{amsthm}
\theoremstyle{definition}
\newtheorem{THM}{Theorem}
\newtheorem{LEM}[THM]{Lemma}
\newtheorem{PROP}[THM]{Proposition}

\newtheorem{DEF}[THM]{Definition}
\newtheorem{RMK}[THM]{Remark}

\newtheorem{QUE}[THM]{Question}
\newtheorem*{THM*}{Theorem}
\newtheorem*{LEM*}{Lemma}
\newtheorem*{PROP*}{Proposition}
\newtheorem*{COR*}{Corollary}
\newtheorem*{DEF*}{Definition}
\newtheorem*{RMK*}{Remark}
\newtheorem*{EX*}{Example}

\numberwithin{figure}{section}
\numberwithin{equation}{section}
\numberwithin{THM}{section}

\newcommand{\TL}{\mathbf{Sp}_q}
\newcommand{\CK}[1]{\mathcal{S}_{(#1)}}
\newcommand{\Kar}{\operatorname{Kar}(\mathbf{Sp}_q)}

\title{A graphical categorification of the two-variable Chebyshev polynomials of the second kind}
\author{Wataru Yuasa}
\date{}
\address{Department of Mathematics\\
  Kyoto University\\
  Kitashirakawa Oiwake-cho, Sakyo-ku, Kyoto 606-8502, Japan}
\email[]{wyuasa@kyoto-u.math.ac.jp}

\thanks{This work was supported by KAKENHI (15H05739).}

\begin{document}
%%%%%%%%%%%%%%%%%%%%
%%%%% abstract %%%%%
%%%%%%%%%%%%%%%%%%%%
\begin{abstract}
We show that the $A_2$ clasps in the Karoubi envelope of $A_2$ spider satisfy the recursive formula of the two-variable Chebyshev polynomials of the second kind associated with a root system of type $A_2$. 
The $A_2$ spider is a diagrammatic description of the representation category for $U_q(\mathfrak{sl}_3)$ and the $A_2$ clasps are projectors.
Our categorification also gives a natural definition of a $q$-deformation of the two-variable Chebyshev polynomials.
This paper is constructed based only on the linear skein theory and graphical calculus.
\end{abstract}
\maketitle

%%%%%%%%%%%%%%%%%%%%%%%%
%%%%% middle arrow %%%%%
%%%%%%%%%%%%%%%%%%%%%%%%
\tikzset{->-/.style={decoration={
  markings,
  mark=at position #1 with {\arrow[black,thin]{>}}},postaction={decorate}}}
\tikzset{-<-/.style={decoration={
  markings,
  mark=at position #1 with {\arrow[black,thin]{<}}},postaction={decorate}}}
\tikzset{-|-/.style={decoration={
  markings,
  mark=at position #1 with {\arrow[black,thin]{|}}},postaction={decorate}}}
%%%%%%%%%%%%%%%%%%%%%%
%%%% triple line %%%%%
%%%%%%%%%%%%%%%%%%%%%%
\tikzset{
    triple/.style args={[#1] in [#2] in [#3]}{
        #1,preaction={preaction={draw,#3},draw,#2}
    }
}

%%%%%%%%%%%%%%%%%%%%%
%%%%% body text %%%%%
%%%%%%%%%%%%%%%%%%%%%
\section{Introduction}
In this paper, we give a categorification of a two-variable Chebyshev polynomial of the second kind inspired by Queffelec and Wedrich~\cite{QueffelecWedrich17A, QueffelecWedrich18A}.
They categorified the Chebyshev polynomials and power-sum symmetric polynomials through diagrammatic categories.
In~\cite{QueffelecWedrich17A}, it is shown that the Jones-Wenzl projectors satisfy the recursive formula of the Chebyshev polynomials of the second kind (resp.~first kind) in the split Grothendieck group of the Karoubi envelope of the Temperley-Lieb category (resp. an affine version of the Temperley-Lieb category). 

From representation theoretical point of view, 
many mathematician and physicist have studied multi-variable generalizations of the Chebyshev polynomials associated with root systems.
For examples, Koornwinder~\cite{Koornwinder74, Koornwinder74b, Koornwinder74c, Koornwinder74d}, Hoffman and Withers~\cite{HoffmanWithers88} for type~$A$, and \cite{NesterenkoPateraszajewskatereszkiewicz10, NesterenkoPateraTereszkiewicz11} in general. 
We consider the following problem:
\begin{QUE}
Give diagrammatic categorifications of multi-variable generalizations of the Chebyshev polynomials.
\end{QUE}

We treat with a two-variable generalization of the Chebyshev polynomials, called the {\em $A_2$ Chebyshev polynomials}, of the second kind appearing in~\cite{Koornwinder74, Koornwinder74b, Koornwinder74c, Koornwinder74d} and give a solution of the above question.  
The $A_2$ Chebyshev polynomials of the second kind is a family of two-variable polynomials $\{\CK{k,l}(x,y)\mid k,l \in\mathbb{Z}_{{}\geq 0}\}$ with integer coefficients. 
It is defined by the following recursive formulas~\cite{Koornwinder74c}, \cite{NesterenkoPateraSzajewskaTereszkiewicz10}:
\begin{align*}
\CK{k+1,l}(x,y)&=x\CK{k,l}(x,y)-\CK{k-1,l+1}(x,y)-\CK{k,l-1}(x,y) \text{ for } k,l\geq 1,\\
\CK{k,l+1}(x,y)&=y\CK{k,l}(x,y)-\CK{k+1,l-1}(x,y)-\CK{k-1,l}(x,y) \text{ for } k,l\geq 1,
\end{align*}
$\CK{k,l}(x,y)=0$ if $k<0$ or $l<0$, and $\CK{0,0}(x,y)=1$.

The categorification of $\CK{k,l}$ is given by using the {\em $A_2$ spider}.
The $A_2$ spider defined in~\cite{Kuperberg96} gives a diagrammatic description of the invariant space $\operatorname{Inv}(V_{\epsilon_1}\otimes V_{\epsilon_2}\otimes\dots\otimes V_{\epsilon_n})$ of the fundamental irreducible representations $V_{+}$ and $V_{-}$ of $U_q(\mathfrak{sl}_3)$ and their tensor products. 
An invariant vector is described as a linear combination of directed uni-trivalent planar graphs with source and sink vertices.
We call it {\em $A_2$ web}. 
By using $A_2$ webs, 
we can consider the $A_2$ version of the Temperley-Lieb category, namely a linear category of intertwining operators between tensor powers of $V_{+}$'s and $V_{-}$'s, through $\operatorname{Hom}(A,B)\cong \operatorname{Inv}(A^{*}\otimes B)$. 
We call it also the $A_2$ spider in this paper, and denote it by $\TL$.
Kuperberg also defined the (internal) $A_2$ clasp $P_{{+}^k{-}^l}^{{+}^k{-}^l}$ in $\operatorname{Hom}(V_{+}^{\otimes k}\otimes V_{-}^{\otimes l},V_{+}^{\otimes k}\otimes V_{-}^{\otimes l})$ for $k,l\geq 0$ as a generalization of the Joens-Wenzl idempotent. 
Ohtsuki and Yamada gave a recursive definition of $P_{{+}^k{-}^l}^{{+}^k{-}^l}$ in~\cite{OhtsukiYamada97}.
Then, $P_{{+}^k{-}^l}^{{+}^k{-}^l}$ give the $A_2$ Chebyshev polynomial in the following sense.
Let $P_{(k,l)}$ be an object in the Karoubi envelope $\Kar$ of $\TL$ corresponding to the idempotent $P_{{+}^k{-}^l}^{{+}^k{-}^l}$.
\begin{THM}\label{mainthm}
The isomorphism classes of $\{P_{(k,l)}\}_{k,l\geq 0}$ satisfy the recursive formula of $\{\CK{k,l}\}_{k,l\geq 0}$ in the split Grothendieck group $K_0(\operatorname{Kar}(\mathbf{Sp}))$ where $\mathbf{Sp}$ is the $A_2$ spider at $q=1$.
\end{THM}

In the above, we consider the $\mathbb{C}$-linear category $\TL$ at $q=1$. 
However, we will show a reccursive formula of $P_{(k,l)}$ for the $\mathbb{C}(q^{\frac{1}{6}})$-linear category $\TL$ and obtain the proof by specializing $q=1$. 
This recursive formula of $P_{(k,l)}$ gives a natural $q$-deformation of the $A_2$ Chebyshev polynomial. 
Many kinds of $q$-deformations of the Chebyshev polynomials have been studied in various contexts, for example, \cite{Dupont10}, \cite{Cigler12A, Cigler12B}.

This paper is organized as follows. 
In Section~2, we recall definitions and properties of the $A_2$ spider and the $A_2$ clasps. 
They are an $A_2$ version of the Temperley-Lieb category and the Jones-Wenzl projectors. 
We categorify the $A_2$ Chebyshev polynomials in Section~3. 
All proofs are given by diagrammatic calculations in the $A_2$ spider.

\section{The $A_2$ spider and the $A_2$ clasp}
In this section, 
we introduce an $A_2$ version of the Temperley-Lieb category $\TL$ based on Kuperberg's $A_2$ spider~\cite{Kuperberg96}. 
Our definition of $\TL$ is only for proof of Theorem~\ref{mainthm}. 
Evans and Pugh studied the details about the $A_2$ version of the Temperley-Lieb category and the $A_2$ planar algebra in \cite{EvansPugh10, EvansPugh11}.
We also define the $A_2$ clasps which play a role analogous to the Jones-Wenzl projectors and show some fundamental properties. 
This section is constructed by the linear skein theory and diagrammatic calculations.
\subsection{The $A_2$ spider}
An {\em $A_2$ web} is a linear combination of graphs in a disk $D=[0,1]\times[0,1]$ with marked points on $[0,1]\times\{0,1\}$ which represents an intertwining operator of $U_q(\mathfrak{sl}_3)$ in manner of Reshetikhin and Turaev~\cite{ReshetikhinTuraev91}.
See a Turaev's book~\cite{Turaev94} and \cite{BakalovKirillov01} for details.
We only give a combinatorial definition of the $A_2$ spider which is a linear category constructed from $A_2$ webs. 
See \cite{Kuperberg94, Kuperberg96} for details on relation to representation theory of $U_q(\mathfrak{sl}_3)$.

Let us recall the $A_2$ web defined by Kuperberg~\cite{Kuperberg96}. 
For any $n\in\mathbb{Z}_{{}\geq 0}$, 
we denote sets of marked points on $[0,1]\times\{0,1\}$ by $P_n=\{\,p_1, p_2, \dots, p_{n}\,\}$ and $Q_n=\{\,q_1, q_2, \dots, q_{n}\,\}$ where $p_i=(i/(n+1),0)$ and $q_i=(i/(n+1),1)$ for $1\leq i\leq n$. 
If $n=0$, then $P_0=Q_0=\emptyset$.
A {\em sign} of $P_{n}$ is a map $\varepsilon_{P_n}\colon P_n\to\{{+},{-}\}$. 
The sign $\varepsilon_{P_n}$ is defined by the sequence $\varepsilon_{P_n(1)}\varepsilon_{P_n(2)}\dots\varepsilon_{P_n(n)}$ of ${+}$ and ${-}$.
A sign of $Q_n$ is defined in the same way.
We consider $D$ with marked points $P_k$ and $Q_l$ with signs $\varepsilon_{P_k}$ and $\varepsilon_{Q_l}$ where $k,l\in\mathbb{Z}_{{}\geq 0}$.
A {\em bipartite uni-trivalent graph} $G$ in $D$ is a directed graph embedded into $D$ such that every vertex is either trivalent or univalent and the vertices are divided into sinks or sources as follows:
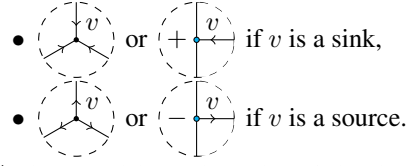
\begin{itemize}
\item \,\tikz[baseline=-.6ex]{
\draw [thin, dashed, fill=white] (0,0) circle [radius=.5];
\draw[-<-=.5] (0:0) -- (90:.5); 
\draw[-<-=.5] (0:0) -- (210:.5); 
\draw[-<-=.5] (0:0) -- (-30:.5);
\node (v) at (0,0) [above right]{$v$};
\fill (0,0) circle [radius=1pt];
}\, 
or 
\,\tikz[baseline=-.6ex]{
\draw [thin, dashed] (0,0) circle [radius=.5];
\clip (0,0) circle [radius=.5];
\draw [thin, fill=white] (0,-.5) rectangle (.5,.5);
\draw[-<-=.5] (0,0)--(.5,0);
\node (p) at (0,0) [left]{${+}$};
\node at (0,0) [above right]{$v$};
\draw [fill=cyan] (0,0) circle [radius=1pt];
}\, if $v$ is a sink,
\item \tikz[baseline=-.6ex]{
\draw [thin, dashed, fill=white] (0,0) circle [radius=.5];
\draw[->-=.5] (0:0) -- (90:.5); 
\draw[->-=.5] (0:0) -- (210:.5); 
\draw[->-=.5] (0:0) -- (-30:.5);
\node (v) at (0,0) [above right]{$v$};
\fill (0,0) circle [radius=1pt];
}\,
or \,\tikz[baseline=-.6ex]{
\draw [thin, dashed] (0,0) circle [radius=.5];
\clip (0,0) circle [radius=.5];
\draw [thin, fill=white] (0,-.5) rectangle (.5,.5);
\draw[->-=.5] (0,0)--(.5,0);
\node (p) at (0,0) [left]{${-}$};
\node at (0,0) [above right]{$v$};
\draw [fill=cyan] (0,0) circle [radius=1pt];
}\, if $v$ is a source.
\end{itemize}
An {\em $A_2$ basis web} is the boundary-fixing isotopy class of a bipartite trivalent graph $G$ in $D$ such that any internal face of $D\setminus G$ has at least six sides. 
Let $\epsilon_k$ and $\epsilon_l$ be sequences of ${+}$ and ${-}$ with length $k$ and $l$, respectively.
We denote the set of $A_2$ basis webs in $D$ with signed marked points $\varepsilon_{P_k}=\epsilon_k$ and $\varepsilon_{Q_l}=\epsilon_l$ by $B(\epsilon_k,\epsilon_l)$.
The {\em $A_2$ web space $W(\epsilon_{k},\epsilon_{l})$} is the $\mathbb{C}(q^{\frac{1}{6}})$-vector space spanned by $B(\epsilon_{k},\epsilon_{l})$.
An {\em $A_2$ web} is an element in the $A_2$ web space. 

For example,
$B{(+-+,-+-)}$ has the following $A_2$ basis webs:
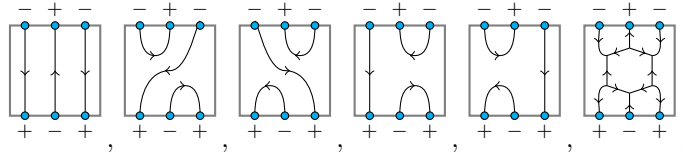
\[
\,\begin{tikzpicture}[scale=.1]
\draw[-<-=.5] (2,0) -- (2,12);
\draw[->-=.5] (6,0) -- (6,12);
\draw[-<-=.5] (10,0) -- (10,12);
\draw[thick, gray] (0,0) rectangle (12,12);
\draw[fill=cyan] (2,0) circle [radius=.5];
\draw[fill=cyan] (6,0) circle [radius=.5];
\draw[fill=cyan] (10,0) circle [radius=.5];
\node at (2,0) [below]{$\scriptstyle{+}$};
\node at (6,0) [below]{$\scriptstyle{-}$};
\node at (10,0) [below]{$\scriptstyle{+}$};
\draw[fill=cyan] (2,12) circle [radius=.5];
\draw[fill=cyan] (6,12) circle [radius=.5];
\draw[fill=cyan] (10,12) circle [radius=.5];
\node at (2,12) [above]{$\scriptstyle{-}$};
\node at (6,12) [above]{$\scriptstyle{+}$};
\node at (10,12) [above]{$\scriptstyle{-}$};
\end{tikzpicture}\,,
\,\begin{tikzpicture}[scale=.1]
\draw[->-=.5] (2,12) to[out=south, in=west] (4,8) to[out=east, in=south] (6,12);
\draw[-<-=.5] (2,0) to[out=north, in=west] (6,6) to[out=east, in=north](10,12);
\draw[->-=.5] (6,0) to[out=north, in=west] (8,4) to[out=east, in=north] (10,0);
\draw[thick, gray] (0,0) rectangle (12,12);
\draw[fill=cyan] (2,0) circle [radius=.5];
\draw[fill=cyan] (6,0) circle [radius=.5];
\draw[fill=cyan] (10,0) circle [radius=.5];
\node at (2,0) [below]{$\scriptstyle{+}$};
\node at (6,0) [below]{$\scriptstyle{-}$};
\node at (10,0) [below]{$\scriptstyle{+}$};
\draw[fill=cyan] (2,12) circle [radius=.5];
\draw[fill=cyan] (6,12) circle [radius=.5];
\draw[fill=cyan] (10,12) circle [radius=.5];
\node at (2,12) [above]{$\scriptstyle{-}$};
\node at (6,12) [above]{$\scriptstyle{+}$};
\node at (10,12) [above]{$\scriptstyle{-}$};
\end{tikzpicture}\,,
\,\begin{tikzpicture}[scale=.1]
\draw[-<-=.5] (2,0) to[out=north, in=west] (4,4) to[out=east, in=north] (6,0);
\draw[-<-=.5] (10,0) to[out=north, in=east] (6,6) to[out=west, in=north](2,12);
\draw[-<-=.5] (6,12) to[out=south, in=west] (8,8) to[out=east, in=south] (10,12);
\draw[thick, gray] (0,0) rectangle (12,12);
\draw[fill=cyan] (2,0) circle [radius=.5];
\draw[fill=cyan] (6,0) circle [radius=.5];
\draw[fill=cyan] (10,0) circle [radius=.5];
\node at (2,0) [below]{$\scriptstyle{+}$};
\node at (6,0) [below]{$\scriptstyle{-}$};
\node at (10,0) [below]{$\scriptstyle{+}$};
\draw[fill=cyan] (2,12) circle [radius=.5];
\draw[fill=cyan] (6,12) circle [radius=.5];
\draw[fill=cyan] (10,12) circle [radius=.5];
\node at (2,12) [above]{$\scriptstyle{-}$};
\node at (6,12) [above]{$\scriptstyle{+}$};
\node at (10,12) [above]{$\scriptstyle{-}$};
\end{tikzpicture}\,,
\,\begin{tikzpicture}[scale=.1]
\draw[-<-=.5] (2,0) -- (2,12);
\draw[->-=.5] (6,0) to[out=north, in=west] (8,4) to[out=east, in=north] (10,0);
\draw[-<-=.5] (6,12) to[out=south, in=west] (8,8) to[out=east, in=south] (10,12);
\draw[thick, gray] (0,0) rectangle (12,12);
\draw[fill=cyan] (2,0) circle [radius=.5];
\draw[fill=cyan] (6,0) circle [radius=.5];
\draw[fill=cyan] (10,0) circle [radius=.5];
\node at (2,0) [below]{$\scriptstyle{+}$};
\node at (6,0) [below]{$\scriptstyle{-}$};
\node at (10,0) [below]{$\scriptstyle{+}$};
\draw[fill=cyan] (2,12) circle [radius=.5];
\draw[fill=cyan] (6,12) circle [radius=.5];
\draw[fill=cyan] (10,12) circle [radius=.5];
\node at (2,12) [above]{$\scriptstyle{-}$};
\node at (6,12) [above]{$\scriptstyle{+}$};
\node at (10,12) [above]{$\scriptstyle{-}$};
\end{tikzpicture}\,,
\,\begin{tikzpicture}[scale=.1]
\draw[-<-=.5] (2,0) to[out=north, in=west] (4,4) to[out=east, in=north] (6,0);
\draw[->-=.5] (2,12) to[out=south, in=west] (4,8) to[out=east, in=south] (6,12);
\draw[-<-=.5] (10,0) -- (10,12);
\draw[thick, gray] (0,0) rectangle (12,12);
\draw[fill=cyan] (2,0) circle [radius=.5];
\draw[fill=cyan] (6,0) circle [radius=.5];
\draw[fill=cyan] (10,0) circle [radius=.5];
\node at (2,0) [below]{$\scriptstyle{+}$};
\node at (6,0) [below]{$\scriptstyle{-}$};
\node at (10,0) [below]{$\scriptstyle{+}$};
\draw[fill=cyan] (2,12) circle [radius=.5];
\draw[fill=cyan] (6,12) circle [radius=.5];
\draw[fill=cyan] (10,12) circle [radius=.5];
\node at (2,12) [above]{$\scriptstyle{-}$};
\node at (6,12) [above]{$\scriptstyle{+}$};
\node at (10,12) [above]{$\scriptstyle{-}$};
\end{tikzpicture}\,,
\,\begin{tikzpicture}[scale=.1]
\draw[-<-=.5] (2,0) to[out=north, in=west] (3,4);
\draw[->-=.5] (6,0) -- (6,3);
\draw[-<-=.5] (10,0) to[out=north, in=east] (9,4);
\draw[->-=.5] (2,12) to[out=south, in=west] (3,8);
\draw[-<-=.5] (6,12) -- (6,9);
\draw[->-=.5] (10,12) to[out=south, in=east] (9,8);
\draw[->-=.5] (3,4) -- (3,8);
\draw[->-=.5] (9,4) -- (9,8);
\draw[-<-=.5] (3,8) -- (6,9);
\draw[->-=.5] (3,4) -- (6,3);
\draw[-<-=.5] (9,8) -- (6,9);
\draw[->-=.5] (9,4) -- (6,3);
\draw[thick, gray] (0,0) rectangle (12,12);
\draw[fill=cyan] (2,0) circle [radius=.5];
\draw[fill=cyan] (6,0) circle [radius=.5];
\draw[fill=cyan] (10,0) circle [radius=.5];
\node at (2,0) [below]{$\scriptstyle{+}$};
\node at (6,0) [below]{$\scriptstyle{-}$};
\node at (10,0) [below]{$\scriptstyle{+}$};
\draw[fill=cyan] (2,12) circle [radius=.5];
\draw[fill=cyan] (6,12) circle [radius=.5];
\draw[fill=cyan] (10,12) circle [radius=.5];
\node at (2,12) [above]{$\scriptstyle{-}$};
\node at (6,12) [above]{$\scriptstyle{+}$};
\node at (10,12) [above]{$\scriptstyle{-}$};
\end{tikzpicture}\,.
\]

\begin{DEF}[The $A_2$ spider]
The {\em $A_2$ spider} $\TL$ is a $\mathbb{C}(q^{\frac{1}{6}})$-linear category defined as follows. 
\begin{itemize}
\item An object of $\TL$ is a finite sequence of signs, that is, a map $\epsilon_n\colon\underline{n}\to\{{+},{-}\}$ where $\underline{n}=\{1<2<\dots<n\}$ is a finite totally ordered set for $n\in\mathbb{Z}_{{}>0}$ where the map $\epsilon_0$ is the map from $\underline{0}=\emptyset$. 
A morphism $\TL(\epsilon_k,\epsilon_l)=\operatorname{Hom}_{\TL}(\epsilon_k,\epsilon_l)$ is the $A_2$ web space $W(\bar{\epsilon}_{k},\epsilon_{l})$ where $\bar{\epsilon}_k$ is the opposite sign of $\epsilon_k$. 
We remark that $\TL(\epsilon_0,\epsilon_0)$ is the $1$-dimensional vector space spanned by the empty web $D_\emptyset$.
\item For $A_2$ basis webs $F\in\TL(\epsilon_k,\epsilon_l)$ and $G\in\TL(\epsilon_l,\epsilon_m)$, the composition $G\circ F=GF\in\TL(\epsilon_k,\epsilon_m)$ is defined by gluing top side of the disk of $F$ and bottom side of the disk of $G$. 
If $GF$ makes $4$-, $2$-, and $0$-gons, we reduce them by
\begin{itemize}
\item 
$\,\tikz[baseline=-.6ex, rotate=90, scale=.8]{
\draw[thin, dashed, fill=white] (0,0) circle [radius=.5];
\draw[->-=.6] (-45:.5) -- (-45:.3);
\draw[->-=.6] (-135:.5) -- (-135:.3);
\draw[-<-=.6] (45:.5) -- (45:.3);
\draw[->-=.6] (135:.5) -- (135:.3);
\draw[->-=.5] (45:.3) -- (135:.3);
\draw[-<-=.5] (-45:.3) -- (-135:.3);
\draw[->-=.5] (45:.3) -- (-45:.3);
\draw[-<-=.5] (135:.3) -- (-135:.3);
}\,
=
\,\tikz[baseline=-.6ex, rotate=90, scale=.8]{
\draw[thin, dashed, fill=white] (0,0) circle [radius=.5];
\draw[->-=.5] (-45:.5) to [out=north west, in=south](.2,0) to [out=north, in=south west](45:.5);
\draw[-<-=.5] (-135:.5) to [out=north east, in=south](-.2,0) to [out=north, in=south east] (135:.5);
}\,
+
\,\tikz[rotate=90, baseline=-.6ex, rotate=90, scale=.8]{
\draw[thin, dashed, fill=white] (0,0) circle [radius=.5];
\draw[-<-=.5] (-45:.5) to [out=north west, in=south](.2,0) to [out=north, in=south west](45:.5);
\draw[->-=.5] (-135:.5) to [out=north east, in=south](-.2,0) to [out=north, in=south east] (135:.5);
}\,
$,\vspace{1ex}
\item 
$\,\tikz[baseline=-.6ex, rotate=-90, scale=.8]{
\draw[thin, dashed, fill=white] (0,0) circle [radius=.5];
\draw[->-=.5] (0,-.5) -- (0,-.25);
\draw[->-=.5] (0,.25) -- (0,.5);
\draw[-<-=.5] (0,-.25) to [out=60, in=120, relative](0,.25);
\draw[-<-=.5] (0,-.25) to [out=-60, in=-120, relative](0,.25);
}\,
=
\left[2\right]\,\tikz[baseline=-.6ex, rotate=-90, scale=.8]{
\draw[thin, dashed, fill=white] (0,0) circle [radius=.5];
\draw[->-=.5] (0,-.5) -- (0,.5);
}\,
$,\vspace{1ex}
\item 
$
\,\tikz[baseline=-.6ex, scale=.8]{
\draw[thin, dashed, fill=white] (0,0) circle [radius=.5];
\draw[->-=.5] (0,0) circle [radius=.3];
}\,
=
\left[3\right]
\,\tikz[baseline=-.6ex, scale=.8]{
\draw[thin, dashed, fill=white] (0,0) circle [radius=.5];
}\,
\quad\,\tikz[baseline=-.6ex, scale=.8]{
\draw[thin, dashed, fill=white] (0,0) circle [radius=.5];
\draw[-<-=.5] (0,0) circle [radius=.3];
}\, ,
=
\left[3\right]
\,\tikz[baseline=-.6ex, scale=.8]{
\draw[thin, dashed, fill=white] (0,0) circle [radius=.5];
}\, ,
$
\end{itemize}
where $[n]=(q^{\frac{n}{2}}-q^{-\frac{n}{2}})/(q^{\frac{1}{2}}-q^{-\frac{1}{2}})$.
\item The identity morphism $\mathbf{1}_{\epsilon_k}$ in $\TL(\epsilon_k,\epsilon_k)$ is the $k$ parallel edges with no trivalent vertex.
\item A tensor product of two objects is defined by their concatenation, that is, $\epsilon_k\otimes\epsilon_{k'}$ is a map $\underline{k+{k'}}\to\{{+},{-}\}$ such that $\epsilon_k\otimes\epsilon_{k'}(i)=\epsilon_k(i)$ if $0\leq i\leq k$ and $\epsilon_k\otimes\epsilon_{k'}(i)=\epsilon_{k'}(i-k)$ if $k+1\leq i\leq k+k'$.
The tensor product of two $A_2$ basis webs $F\in\TL(\epsilon_k,\epsilon_l)$ and $G\in\TL(\epsilon_k',\epsilon_l')$ is defined by gluing right side of the disk of $F$ and left side of the disk of $G$.
\end{itemize}
\end{DEF}

We remark that an $A_2$ web $f\in\TL(\epsilon_k,\epsilon_l)$ represents an intertwining operator in $\operatorname{Hom}(V_k,V_l)$ where $V_k=V_{\epsilon_{k}(1)}\otimes V_{\epsilon_{k}(2)}\otimes\dots\otimes V_{\epsilon_{k}(k)}$ and $V_l=V_{\epsilon_{l}(1)}\otimes V_{\epsilon_{l}(2)}\otimes\dots\otimes V_{\epsilon_{l}(l)}$.
As an invariant vector $f$ is an element in $\operatorname{Inv}(V_{k}^{*}\otimes V_{l})$ and
$V_{k}^{*}\cong V_{\bar{\epsilon}_k(k)}\otimes\dots\otimes V_{\bar{\epsilon}_k(2)}\otimes V_{\bar{\epsilon}_k(1)}$ because $V_{\pm}^*\cong V_{\mp}$. 
See \cite{Kuperberg96} in detail. 
The signs of marked points of the $A_2$ web space are labeled according to the signs of the invariant space.

\subsection{$A_2$ clasps}
We give a diagrammatic definition of an $A_2$ version of the Jones-Wenzl projectors, called the $A_2$ clasps, introduced in \cite{Kuperberg96, OhtsukiYamada97}. 
The purpose of this section is to construct the general form of the $A_2$ clasp and to prove its properties by only using the linear skein theory.

In what follows,
we omit ``${\otimes}$'' to describe the tensor product of objects in $\TL$ and ${+}$ (resp.~${-}$) means the constant map from $\underbar{1}$ to ${+}$ (resp.~${-}$). 
Thus, ${+}^k$ (resp.~${-}^k$) is the constant map from $\underbar{k}$ to ${+}$ (resp.~${-}$) for any positive integer $k$. 
\begin{DEF}[The $A_2$ clasp in $\TL({+}^{k},{+}^{k})$]\label{singledef}
\ 
\begin{enumerate}
\item
$\tikz[baseline=-.6ex, scale=.1]{
\draw[->-=.8] (0,-4) -- (0,4);
\draw[fill=white] (-3,-1) rectangle (3,1);
\node at (0,3) [left]{${\scriptstyle 1}$};
}\,
= 
\,\tikz[baseline=-.6ex, scale=.1]{
\draw[->-=.5] (0,-4) -- (0,4); 
}\,
=\mathbf{1}_{+}\quad\in \TL(+,+)$
\item
$\tikz[baseline=-.6ex, scale=.1]{
\draw[->-=.8] (0,-4) -- (0,4);
\draw[fill=white] (-3,-1) rectangle (3,1);
\node at (0,3) [left]{${\scriptstyle k}$};
}\,
= 
\,\tikz[baseline=-.6ex, scale=.1]{
\draw[->-=.8] (0,-4) -- (0,4);
\draw[->-=.5] (4,-4) -- (4,4);
\draw[fill=white] (-2,-.5) rectangle (2,.5);
\node at (0,3) [left]{${\scriptstyle k-1}$};
}\,
-\frac{[k-1]}{[k]}
\,\tikz[baseline=-.6ex, scale=.1]{
\draw[->-=.6] (1,3) -- (1,6);
\draw[->-=.5] (1,-6) -- (1,-3);
\draw[->-=.5] (0,-3) -- (0,3);
\draw[-<-=.5] (4,-1) -- (4,1);
\draw[-<-=.7] (2,3) to[out=south, in=west] (4,1);
\draw[-<-=.5] (6,6) -- (6,3) to[out=south, in=east] (4,1);
\draw[->-=.7, yscale=-1] (2,3) to[out=south, in=west] (4,1);
\draw[->-=.5, yscale=-1] (6,6) -- (6,3) to[out=south, in=east] (4,1);
\draw[fill=white, yshift=3cm] (-2,-.5) rectangle (3,.5);
\draw[fill=white, yshift=-3cm] (-2,-.5) rectangle (3,.5);
\node at (1,6) [left]{${\scriptstyle k-1}$};
\node at (1,-6) [left]{${\scriptstyle k-1}$};
\node at (0,0) [left]{${\scriptstyle k-2}$};
}\,\quad\in\TL({+}^k,{+}^k)$
\end{enumerate}
The above recursive formula defines the $A_2$ clasp in $\TL({+}^{k},{+}^{k})$ and we denote it by $P_{{+}^k}^{{+}^k}$.
The $A_2$ clasp in $\TL({-}^k,{-}^k)$ is also defined by the same way and we denote it by $P_{{-}^k}^{{-}^k}$.
\end{DEF}

We introduce the following $A_2$ webs:
\begin{equation}\label{tbd}
t_{-}^{++}
=
\,\tikz[baseline=-.6ex, scale=.1]{
\draw[->-=.5] (0:0) -- (45:4);
\draw[->-=.5] (0:0) -- (135:4);
\draw[->-=.5] (0:0) -- (270:3);
}\, , t_{+}^{--}
=
\,\tikz[baseline=-.6ex, scale=.1]{
\draw[-<-=.5] (0:0) -- (45:4);
\draw[-<-=.5] (0:0) -- (135:4);
\draw[-<-=.5] (0:0) -- (270:3);
}\,
, t_{--}^{+}
=
\,\tikz[baseline=-.6ex, scale=.1, yscale=-1]{
\draw[->-=.5] (0:0) -- (45:4);
\draw[->-=.5] (0:0) -- (135:4);
\draw[->-=.5] (0:0) -- (270:3);
}\,
, t_{++}^{-}
=
\,\tikz[baseline=-.6ex, scale=.1, yscale=-1]{
\draw[-<-=.5] (0:0) -- (45:4);
\draw[-<-=.5] (0:0) -- (135:4);
\draw[-<-=.5] (0:0) -- (270:3);
}\, ,
\end{equation}
\begin{equation*}
b^{-+}
=
\,\tikz[baseline=-.6ex, scale=.1]{
\draw[->-=.5] (-2,2) to[out=south, in=west] (0,-1) to[out=east, in=south] (2,2);
}\,
, b^{+-}
=
\,\tikz[baseline=-.6ex, scale=.1]{
\draw[-<-=.5] (-2,2) to[out=south, in=west] (0,-1) to[out=east, in=south] (2,2);
}\,
, d_{+-}
=
\,\tikz[baseline=-.6ex, scale=.1, yscale=-1]{
\draw[->-=.5] (-2,2) to[out=south, in=west] (0,-1) to[out=east, in=south] (2,2);
}\,
, d_{-+}
=
\,\tikz[baseline=-.6ex, scale=.1, yscale=-1]{
\draw[-<-=.5] (-2,2) to[out=south, in=west] (0,-1) to[out=east, in=south] (2,2);
}\,
\end{equation*}

An $A_2$ basis web $B(\bar{\epsilon}_{k},\epsilon_{k})$ provides a tiling of $D$. 
The following Lemma is easily shown by calculating the Euler number of the tiling.
\begin{LEM}[Ohtsuki and Yamada~{\cite[Lemma~3.3]{OhtsukiYamada97}}]\label{decomp}
For any $A_2$ basis web $D$ in $\TL(\epsilon_k,\epsilon_k)$ other than $\mathbf{1}_{\epsilon_k}$ has $t_{{-}}^{{+}{+}}$, $t_{{+}}^{{-}{-}}$, $b^{{-}{+}}$, or $b^{{+}{-}}$ in the top side of $D$ and $t_{{-}{-}}^{{+}}$, $t_{{+}{+}}^{{-}}$, $d_{{-}{+}}$, or $d_{{+}{-}}$ in the bottom side of $D$.
\end{LEM}

\begin{PROP}[Kuperberg~\cite{Kuperberg96}, Ohtsuki and Yamada~\cite{OhtsukiYamada97}]\label{singleproj} 
Let $k$ be a non-negative integer, then
\begin{enumerate}
\item $P_{{+}^k}^{{+}^k}(\mathbf{1}_{{+}^a}\otimes P_{{+}^l}^{{+}^l}\otimes\mathbf{1}_{{+}^b})=P_{{+}^k}^{{+}^k}=(\mathbf{1}_{{+}^a}\otimes P_{{+}^l}^{{+}^l}\otimes\mathbf{1}_{{+}^b})P_{{+}^k}^{{+}^k}$ for $a+b+l=k$,
\item $P_{{+}^k}^{{+}^k}(\mathbf{1}_{{+}^a}\otimes t_{-}^{++}\otimes\mathbf{1}_{{+}^b})=0=(\mathbf{1}_{{+}^a}\otimes t_{++}^{-}\otimes \mathbf{1}_{{+}^b})P_{{+}^k}^{{+}^k}$ for $a+b+2=k$.
\end{enumerate}
The above is true for the opposite sign.
\end{PROP}

\begin{PROP}[Uniqueness]\label{singleunique}
If a non-trivial element $T\in\TL({\epsilon}^k,{\epsilon}^k)$ satisfies $T^2=T$ and Proposition~\ref{singleproj}~(2), then $T=P_{{\epsilon}^k}^{{\epsilon}^k}$ for $\epsilon\in\{+,-\}$. 
\end{PROP}
\begin{proof}
$P_{{+}^k}^{{+}^k}$ can be expanded as a linear sum of $A_2$ basis webs $P_{{+}^k}^{{+}^k}=c\mathbf{1}_{{+}^k}+x$ by Lemma~{\ref{decomp}} where $c$ is a constant and $x$ have $t_{-}^{++}$ and $t_{++}^{-}$ in the top side and in the bottom side.
We can see that $c=1$ from $(P_{{+}^k}^{{+}^k})^2=P_{{+}^k}^{{+}^k}$ and Proposition~\ref{singleproj}~(2). 
In the same way, we know $T=\mathbf{1}_{{+}^k}+x'$.
Therefore, 
\begin{align*}
P_{{+}^k}^{{+}^k}T&=(\mathbf{1}_{{+}^k}+x)T=T\\
P_{{+}^k}^{{+}^k}T&=P_{{+}^k}^{{+}^k}(\mathbf{1}_{{+}^k}+x')=P_{{+}^k}^{{+}^k}
\end{align*}
We can prove it for $\epsilon={-}$ in the same way.
\end{proof}

For an $A_2$ basis web $w\in B(\bar{\epsilon}_{k},\epsilon_{l})$, we define $w^*\in B(\bar{\epsilon}_{l},\epsilon_{k})$ as the reflection of $w$ through the horizontal line $[0,1]\times\{1/2\}$ with the opposite direcion.
For the coefficient $\mathbb{C}(q^\frac{1}{6})$, the star operator acts on $\mathbb{C}$ by complex conjugate and $(q^{\frac{1}{6}})^*=q^{-\frac{1}{6}}$. 
In this way, we define a linear map ${}^*\colon \TL(\epsilon_k,\epsilon_l)\to\TL(\epsilon_{l},\epsilon_{k})$.

\begin{RMK}\label{singlerem}
Proposition~\ref{singleunique} implies $(P_{{+}^k}^{{+}^k})^*=P_{{+}^k}^{{+}^k}$ and $(P_{{-}^k}^{{-}^k})^*=P_{{-}^k}^{{-}^k}$.
\end{RMK}

We introduce the $A_2$ clasp $P_{{+}^k{-}^l}^{{+}^k{-}^l}$ in $\TL({+}^k{-}^l,{+}^k{-}^l)$ based on~\cite{OhtsukiYamada97}.
\begin{DEF}[The $A_2$ clasp in $\TL({+}^k{-}^l,{+}^k{-}^l)$]\label{OYdef}
\[P_{{+}^k{-}^l}^{{+}^k{-}^l}
=
\tikz[baseline=-.6ex, scale=.1]{
\draw[->-=.2, ->-=.8] (-2,-5) -- (-2,5);
\draw[-<-=.2, -<-=.8] (2,-5) -- (2,5);
\draw[fill=white] (-4,-1) rectangle (4,1);
\node at (-2,5) [left]{${\scriptstyle k}$};
\node at (-2,-5) [left]{${\scriptstyle k}$};
\node at (2,5) [right]{${\scriptstyle l}$};
\node at (2,-5) [right]{${\scriptstyle l}$};
}\,
=
\sum_{i=0}^{\min\{k,l\}}
(-1)^i
\frac{{k\brack i}{l\brack i}}{{k+l+1\brack i}}
\,\tikz[baseline=-.6ex, scale=.1]{
\draw[->-=.5] (-4,5) -- (-4,8);
\draw[-<-=.5] (4,5) -- (4,8);
\draw[-<-=.5, yscale=-1] (-4,5) -- (-4,8);
\draw[->-=.5, yscale=-1] (4,5) -- (4,8);
\draw[->-=.5] (-5,-4) -- (-5,4);
\draw[-<-=.5] (5,-4) -- (5,4);
\draw[-<-=.5] (-3,4) to[out=south, in=south] (3,4);
\draw[->-=.5, yscale=-1] (-3,4) to[out=south, in=south] (3,4);
\draw[fill=white] (-6,4) rectangle (-2,5);
\draw[fill=white, xscale=-1] (-6,4) rectangle (-2,5);
\draw[fill=white, yscale=-1] (-6,4) rectangle (-2,5);
\draw[fill=white, xscale=-1, yscale=-1] (-6,4) rectangle (-2,5);
\node at (-5,0) [left]{${\scriptstyle k-i}$};
\node at (5,0) [right]{${\scriptstyle l-i}$};
\node at (0,4) {${\scriptstyle i}$};
\node at (0,-4) {${\scriptstyle i}$};
\node at (-4,7) [left]{${\scriptstyle k}$};
\node at (4,7) [right]{${\scriptstyle l}$};
\node at (-4,-7) [left]{${\scriptstyle k}$};
\node at (4,-7) [right]{${\scriptstyle l}$};
}\, .\]
$P_{{-}^k{+}^l}^{{-}^k{+}^l}$ is also defined by the same way.
\end{DEF}

One can prove a similar statement to Proposition~\ref{singleproj}.

\begin{PROP}[Kuperberg~\cite{Kuperberg96}, Ohtsuki and Yamada~\cite{OhtsukiYamada97}]\label{doubleproj} 
Let $k$ and $l$ be non-negative integers, then
\begin{enumerate}
\item $P_{{+}^k{-}^l}^{{+}^k{-}^l}
(\mathbf{1}_{{+}^a}\otimes P_{{+}^s{-}^t}^{{+}^s{-}^t}\otimes\mathbf{1}_{{-}^b})
=P_{{+}^k{-}^l}^{{+}^k{-}^l}
=(\mathbf{1}_{{+}^a}\otimes P_{{+}^s{-}^t}^{{+}^s{-}^t}\otimes\mathbf{1}_{{-}^b})
P_{{+}^k{-}^l}^{{+}^k{-}^l}$ for $a+s=k$ and $b+t=l$,
\item $P_{{+}^k{-}^l}^{{+}^k{-}^l}(\mathbf{1}_{{+}^a}\otimes b^{+-}\otimes\mathbf{1}_{{-}^b})=0=(\mathbf{1}_{{+}^a}\otimes d_{+-}\otimes \mathbf{1}_{{-}^b})P_{{+}^k{-}^l}^{{+}^k{-}^l}$ for $a+1=k$ and $b+1=l$.
\end{enumerate}
The same equalities hold for $P_{{-}^k{+}^l}^{{-}^k{+}^l}$. 
\end{PROP}

One can prove the uniqueness of $P_{{+}^k{-}^l}^{{+}^k{-}^l}$ in a similar way to Proposition~\ref{singleunique}.

\begin{PROP}[Uniqueness]\label{doubleunique}
A non-trivial idempotent element in $\TL({+}^k{-}^l,{+}^k{-}^l)$ satisfying Proposition~\ref{doubleproj}~(2) is uniquely determined.
\end{PROP}
\begin{proof}
In the same way as the proof of Proposition~\ref{singleunique}.
\end{proof}

\begin{RMK}\label{doublerem}
Proposition~\ref{doubleunique} implies $(P_{{+}^k{-}^l}^{{+}^k{-}^l})^*=P_{{+}^k{{-}^l}}^{{+}^k{-}^l}$ and $(P_{{-}^k{+}^l}^{{-}^k{+}^l})^*=P_{{-}^k{+}^l}^{{-}^k{+}^l}$.
\end{RMK}

We give an explicit definition of a general form of the $A_2$ clasp appear in Kuperberg~\cite{Kuperberg96} and Kim~\cite{Kim07}. 
This $A_2$ clasp is an $A_2$ web, no longer idempotent, in $\TL(\epsilon,\delta)$ such that $k=\#\epsilon^{-1}({+})=\#{\delta}^{-1}({+})$ and $l=\#\epsilon^{-1}({-})=\#{\delta}^{-1}({-})$.
We introduce the following $A_2$ basis webs:

\begin{equation}\label{H}
H_{{+}{-}}^{{-}{+}}
=
\,\tikz[baseline=-.6ex, scale=.1]{
\draw[->-=.2,-<-=.8] (0,-3) -- (0,3);
\draw[-<-=.2,->-=.8] (4,-3) -- (4,3);
\draw[-<-=.5] (0,0) -- (4,0);
}\,
,
\quad H_{{-}{+}}^{{+}{-}}
=
\,\tikz[baseline=-.6ex, scale=.1]{
\draw[-<-=.2,->-=.8] (0,-3) -- (0,3);
\draw[->-=.2,-<-=.8] (4,-3) -- (4,3);
\draw[->-=.5] (0,0) -- (4,0);
}\, .
\end{equation}

Let $k$ and $l$ be non-negative integers. 
We take an arbitrary object $\epsilon\colon\underline{k+l}\to\{\pm\}$ in $\TL$ satisfying $k=\#\epsilon^{-1}(+)$ and $l=\#\epsilon^{-1}(-)$. 
Then, we define an $A_2$ basis web $\sigma_{{+}^k{-}^l}^{\epsilon}$ as follows.
We consider the disk $D=[0,1]\times[0,1]$ with marked points signed by ${+}^k{-}^l$ and $\epsilon$.
Join the marked points labeled by ${+}$ in the bottom side with ones of the upper side by straight arcs.
In the same way, join the mark points labeled by ${-}$ by straight arcs. 
Then, one can obtain the $A_2$ basis web $\sigma_{{+}^k{-}^l}^{\epsilon}$ by replacing all crossing points by $H_{{+}{-}}^{{-}{+}}$.

\begin{DEF}
Let $k$ and $l$ be non-negative integers. 
Then, an $A_2$ clasp $P_{{+}^k{-}^l}^{\epsilon}$ in $\TL({{+}^k{-}^l},{\epsilon})$ is defined by
\[
 P_{{+}^k{-}^l}^{\epsilon}=\sigma_{{+}^k{-}^l}^{\epsilon}P_{{+}^k{-}^l}^{{+}^k{-}^l}.
\]
\end{DEF}

\begin{PROP}
Compositions of $P_{{+}^k{-}^l}^{\epsilon}$ with $A_2$ basis webs $t$ and $d$ vanish.
\end{PROP}

\begin{proof}
We prove it by induction on the number $h(\sigma)$ of $H_{{+}{-}}^{{-}{+}}$ contained in $\sigma=\sigma_{{+}^k{-}^l}^{\epsilon}$.
If $h(\sigma)=0$, it is clear since $P_{{+}^k{-}^l}^{\epsilon}=P_{{+}^k{-}^l}^{{+}^k{-}^l}$ and Proposition~\ref{doubleproj}. 
If $h(\sigma)=1$, then $\sigma=\mathbf{1}_{{+}^{k-1}}\otimes H_{{+}{-}}^{{-}{+}}\mathbf{1}_{{-}^{l-1}}$. 
One can prove by easy calculations.
When $h(\sigma)=n+1$ $(n\geq 1)$, 
$\sigma$ is described as a composition of $\sigma'$ with $\mathbf{1}_{\alpha}\otimes H_{{+}{-}}^{{-}{+}}\otimes\mathbf{1}_{\beta}$ whrere $\epsilon=\alpha{-}{+}\beta$ and $\sigma'=\sigma_{{+}^k{-}^l}^{\alpha{+}{-}\beta}$ such that $h(\sigma')=n$.
Thus, $\sigma$ is
$
\,\tikz[baseline=-.6ex, scale=.1]{
\draw (-8,-2) -- (-8,2);
\node at (-6,0) {$\scriptstyle{\cdots}$};
\draw[->-=.5] (-4,-2) -- (-4,2);
\draw[->-=.3,-<-=.8] (-2,-2) -- (-2,2);
\draw[-<-=.3,->-=.8] (2,-2) -- (2,2);
\draw[-<-=.5] (-2,0) -- (2,0);
\draw[-<-=.5] (4,-2) -- (4,2);
\node at (6,0) {$\scriptstyle{\cdots}$};
\draw (8,-2) -- (8,2);
\draw[thick] (-9,-2) rectangle (9,-5);
\node at (0,-3.5) {$\scriptstyle{\sigma'}$};
}\,
$
,
$
\,\tikz[baseline=-.6ex, scale=.1]{
\draw (-8,-2) -- (-8,2);
\node at (-6,0) {$\scriptstyle{\cdots}$};
\draw[->-=.5] (-4,-2) -- (-4,2);
\draw[->-=.3,-<-=.8] (-2,-2) -- (-2,2);
\draw[-<-=.3,->-=.8] (2,-2) -- (2,2);
\draw[-<-=.5] (-2,0) -- (2,0);
\draw[->-=.5] (4,-2) -- (4,2);
\node at (6,0) {$\scriptstyle{\cdots}$};
\draw (8,-2) -- (8,2);
\draw[thick] (-9,-2) rectangle (9,-5);
\node at (0,-3.5) {$\scriptstyle{\sigma'}$};
}\,
$
,
$
\,\tikz[baseline=-.6ex, scale=.1]{
\draw (-8,-2) -- (-8,2);
\node at (-6,0) {$\scriptstyle{\cdots}$};
\draw[-<-=.5] (-4,-2) -- (-4,2);
\draw[->-=.3,-<-=.8] (-2,-2) -- (-2,2);
\draw[-<-=.3,->-=.8] (2,-2) -- (2,2);
\draw[-<-=.5] (-2,0) -- (2,0);
\draw[->-=.5] (4,-2) -- (4,2);
\node at (6,0) {$\scriptstyle{\cdots}$};
\draw (8,-2) -- (8,2);
\draw[thick] (-9,-2) rectangle (9,-5);
\node at (0,-3.5) {$\scriptstyle{\sigma'}$};
}\,
$
or
$
\,\tikz[baseline=-.6ex, scale=.1]{
\draw (-8,-2) -- (-8,2);
\node at (-6,0) {$\scriptstyle{\cdots}$};
\draw[-<-=.5] (-4,-2) -- (-4,2);
\draw[->-=.3,-<-=.8] (-2,-2) -- (-2,2);
\draw[-<-=.3,->-=.8] (2,-2) -- (2,2);
\draw[-<-=.5] (4,-2) -- (4,2);
\draw[-<-=.5] (-2,0) -- (2,0);
\node at (6,0) {$\scriptstyle{\cdots}$};
\draw (8,-2) -- (8,2);
\draw[thick] (-9,-2) rectangle (9,-5);
\node at (0,-3.5) {$\scriptstyle{\sigma'}$};
}\,
$
.
In the first case, 
one can show by easy calculations and the induction hypothesis. 
In the second case, 
we only have to prove 
$
\,\tikz[baseline=-.6ex, scale=.1]{
\draw (-8,-2) -- (-8,2);
\node at (-6,0) {$\scriptstyle{\cdots}$};
\draw[->-=.5] (-4,-2) -- (-4,2);
\draw[->-=.3,-<-=.8] (-2,-2) -- (-2,2);
\draw[-<-=.3,->-=.8] (2,-2) -- (2,0) to[out=north, in=west] (3,2);
\draw[-<-=.5] (-2,0) -- (2,0);
\draw[->-=.5] (4,-2) -- (4,0) to[out=north, in=east] (3,2);
\node at (6,0) {$\scriptstyle{\cdots}$};
\draw (8,-2) -- (8,2);
\draw[thick] (-9,-2) rectangle (9,-5);
\node at (0,-3.5) {$\scriptstyle{\sigma'}$};
\draw[-<-=.5] (3,2) -- (3,4);
}\,
\circ P_{{+}^k{-}^l}^{{+}^k{-}^l}=0.
$
By construction of $\sigma$, the right leg of $H_{{+}{-}}^{{-}{+}}$ and the up-pointing arc on one's right should have a crossing ($H_{{+}{-}}^{{-}{+}}$).
Then, there exists $\sigma''$ such that $h(\sigma'')=n-1$ and 
\begin{align*}
\,\tikz[baseline=-.6ex, scale=.1]{
\draw (-8,-2) -- (-8,2);
\node at (-6,0) {$\scriptstyle{\cdots}$};
\draw[->-=.5] (-4,-2) -- (-4,2);
\draw[->-=.3,-<-=.8] (-2,-2) -- (-2,2);
\draw[->-=.8] (2,-2) -- (2,0) to[out=north, in=west] (3,2);
\draw[-<-=.5] (-2,0) -- (2,0);
\draw[->-=.5] (4,-2) -- (4,0) to[out=north, in=east] (3,2);
\node at (6,0) {$\scriptstyle{\cdots}$};
\draw (8,-2) -- (8,2);
\draw[thick] (-9,-2) rectangle (9,-5);
\node at (0,-3.5) {$\scriptstyle{\sigma''}$};
\draw[-<-=.5] (3,2) -- (3,4);
\draw[-<-=.5] (2,-1) -- (4,-1);
}\,
=
\,\tikz[baseline=-.6ex, scale=.1]{
\draw (-8,-2) -- (-8,2);
\node at (-6,0) {$\scriptstyle{\cdots}$};
\draw[->-=.5] (-4,-2) -- (-4,2);
\draw[->-=.3,-<-=.8] (-2,-2) -- (-2,2);
\draw[->-=.8] (2,-2) -- (2,0);
\draw[-<-=.5] (-2,0) -- (2,0);
\draw[-<-=.5] (4,-2) -- (4,0) to[out=north, in=south] (3,2);
\node at (6,0) {$\scriptstyle{\cdots}$};
\draw (8,-2) -- (8,2);
\draw[thick] (-9,-2) rectangle (9,-5);
\node at (0,-3.5) {$\scriptstyle{\sigma''}$};
}\,
+
\,\tikz[baseline=-.6ex, scale=.1]{
\draw (-8,-2) -- (-8,2);
\node at (-6,0) {$\scriptstyle{\cdots}$};
\draw[->-=.5] (-4,-2) -- (-4,2);
\draw[->-=.3,-<-=.8] (-2,-2) -- (-2,2);
\draw[-<-=.5] (2,-2) to[out=north, in=west] (3,0) to[out=east, in=north] (4,-2);
\draw[-<-=.5] (-2,0) -- (1,0) to[out=east, in=south] (3,2);
\node at (6,0) {$\scriptstyle{\cdots}$};
\draw (8,-2) -- (8,2);
\draw[thick] (-9,-2) rectangle (9,-5);
\node at (0,-3.5) {$\scriptstyle{\sigma''}$};
}\,.
\end{align*}
These terms has $t_{{+}{+}}^{-}$ and $d_{{-}{+}}$ on the web $\sigma''$. 
Therefore, 
$
\,\tikz[baseline=-.6ex, scale=.1]{
\draw (-8,-2) -- (-8,2);
\node at (-6,0) {$\scriptstyle{\cdots}$};
\draw[->-=.5] (-4,-2) -- (-4,2);
\draw[->-=.3,-<-=.8] (-2,-2) -- (-2,2);
\draw[-<-=.3,->-=.8] (2,-2) -- (2,0) to[out=north, in=west] (3,2);
\draw[-<-=.5] (-2,0) -- (2,0);
\draw[->-=.5] (4,-2) -- (4,0) to[out=north, in=east] (3,2);
\node at (6,0) {$\scriptstyle{\cdots}$};
\draw (8,-2) -- (8,2);
\draw[thick] (-9,-2) rectangle (9,-5);
\node at (0,-3.5) {$\scriptstyle{\sigma'}$};
\draw[-<-=.5] (3,2) -- (3,4);
}\,
\circ P_{{+}^k{-}^l}^{{+}^k{-}^l}=0,
$
by the induction hypothesis for $\sigma''\circ P_{{+}^k{-}^l}^{{+}^k{-}^l}$.
For other cases, we can prove in the same way.
\end{proof}

\begin{PROP}\label{multiprop} 
Let us decompose $\epsilon$ into three subsequences $\alpha\beta\gamma$ such that $\beta$ is expressed as the form ${{+}{+}\cdots{+}{-}{-}\cdots{-}}$ or ${{-}{-}\cdots{-}{+}{+}\cdots{+}}$. 
Then,
\begin{enumerate}
\item $P_{{+}^k{-}^l}^{\epsilon}=(\mathbf{1}_{\alpha}\otimes P_{\beta}^{\beta}\otimes\mathbf{1}_{\gamma})
P_{{+}^k{-}^l}^{\epsilon}$,
\item $(P_{{+}^k{-}^l}^{\epsilon})^{*}P_{{+}^k{-}^l}^{\epsilon}=P_{{+}^k{-}^l}^{{+}^k{-}^l}$.
\end{enumerate}
Especially, $tP_{{+}^k{-}^l}^{\epsilon}=0$ and $dP_{{+}^k{-}^l}^{\epsilon}=0$ where $t$ (resp.~$d$) is a tensor product of identity morphisms and $t_{{-}{-}}^{{+}}$ or $t_{{+}{+}}^{{-}}$ (resp. $d_{{+}{-}}$ or $d_{{-}{+}}$).
\end{PROP}
\begin{proof}
(1) is easily shown by expanding $P_{\beta}^{\beta}$.
We can desctibe $P_{{+}^k{-}^l}^{\epsilon}$ as a product $\tau P_{{+}^k{-}^l}^{\epsilon'}$ where $\tau$ is a tensor product of identity morphisms and only one $H_{{+}{-}}^{{-}{+}}$ or $H_{{-}{+}}^{{+}{-}}$.
Then, $(P_{{+}^k{-}^l}^{\epsilon})^{*}P_{{+}^k{-}^l}^{\epsilon}=(P_{{+}^k{-}^l}^{\epsilon'})^{*}\tau^*\tau P_{{+}^k{-}^l}^{\epsilon'}$ and one can finish the proof of (2) by applying the defining relation of the $A_2$ web to $\tau^*\tau$.
\end{proof}

Let us define an $A_2$ clasp $P_{\epsilon}^{\epsilon}$ in $\TL(\epsilon,\epsilon)$ by $P_{{+}^k{-}^l}^{\epsilon}(P_{{+}^k{-}^l}^{\epsilon})^{*}$.
Then,
\begin{PROP}\label{eeprop}\ 
\begin{enumerate}
\item $(P_{\epsilon}^{\epsilon})^2=P_{\epsilon}^{\epsilon}$,
\item $tP_{\epsilon}^{\epsilon}$, $dP_{\epsilon}^{\epsilon}$, $P_{\epsilon}^{\epsilon}b$, and $P_{\epsilon}^{\epsilon}t'$ vanish.
\end{enumerate}
In the above, $t'$ (resp.~$b$) is a tensor product of identity morphisms and $t_{{-}}^{{+}{+}}$ or $t_{{+}}^{{-}{-}}$ (resp. $b^{{-}{+}}$ or $b^{{+}{-}}$).
\end{PROP}
\begin{proof}
\begin{align*}
(P_{\epsilon}^{\epsilon})^2
&=P_{{+}^k{-}^l}^{\epsilon}(P_{{+}^k{-}^l}^{\epsilon})^{*}P_{{+}^k{-}^l}^{\epsilon}(P_{{+}^k{-}^l}^{\epsilon})^{*}\\
&=P_{{+}^k{-}^l}^{\epsilon}P_{{+}^k{-}^l}^{{+}^k{-}^l}(P_{{+}^k{-}^l}^{\epsilon})^{*}\\
&=P_{{+}^k{-}^l}^{\epsilon}(P_{{+}^k{-}^l}^{\epsilon})^{*}=P_{\epsilon}^{\epsilon}
\end{align*}
\end{proof}

\begin{PROP}[Uniqueness]\label{eeunique}
If $T$ in $\TL(\epsilon,\epsilon)$ satisfies the conditions (1) and (2) of Proposition~\ref{eeprop}, then $T=P_{\epsilon}^{\epsilon}$.
\end{PROP}
\begin{proof}
In the same way as the proof of Proposition~\ref{singleunique}.
\end{proof}

\begin{RMK}
Proposition~\ref{eeunique} implies $(P_{\epsilon}^{\epsilon})^{*}=P_{\epsilon}^{\epsilon}$.
\end{RMK}

For any non-negative integers $k$ and $l$, 
we consider a subset $\TL^{(k,l)}=\{\,\epsilon\in\TL\mid k=\#\epsilon^{-1}({+}), l=\#\epsilon^{-1}({-})\,\}$ of $\TL$. 
\begin{PROP}\label{generalprop}
There exist a set of $A_2$ webs $\{\,P_{\alpha}^{\beta}\in\TL(\alpha,\beta)\mid \alpha,\beta\in\TL^{(k,l)}\,\}$ satisfying
\begin{enumerate}
\item $P_{\beta}^{\gamma}P_{\alpha}^{\beta}=P_{\alpha}^{\gamma}$ for any $\alpha,\beta,\gamma\in\TL^{(k,l)}$,
\item $tP_{\alpha}^{\beta}$, $dP_{\alpha}^{\beta}$, $P_{\alpha}^{\beta}b$, and $P_{\alpha}^{\beta}t'$ vanish for any $\alpha,\beta\in\TL^{(k,l)}$.
\end{enumerate}
\end{PROP}
\begin{proof}
Let us define $P_{\alpha}^{\beta}=P_{{+}^k{-}^l}^{\beta}(P_{{+}^k{-}^l}^{\alpha})^{*}$ for any $\alpha,\beta,\gamma\in\TL^{(k,l)}$. 
Then, it is obvious that $P_{\alpha}^{\beta}$ satisfies (2) because of Proposition~\ref{multiprop}~(1).
We show the equation of (1):
\begin{align*}
P_{\beta}^{\gamma}P_{\alpha}^{\beta}
&=P_{{+}^k{-}^l}^{\gamma}(P_{{+}^k{-}^l}^{\beta})^{*}P_{{+}^k{-}^l}^{\beta}(P_{{+}^k{-}^l}^{\alpha})^{*}\\
&=P_{{+}^k{-}^l}^{\gamma}P_{{+}^k{-}^l}^{{+}^k{-}^l}(P_{{+}^k{-}^l}^{\alpha})^{*}\quad\text{ by Proposition~\ref{multiprop}~(2)}\\
&=P_{{+}^k{-}^l}^{\gamma}(\mathbf{1}_{{+}^k{-}^l}+x)(P_{{+}^k{-}^l}^{\alpha})^{*}\\
&=P_{{+}^k{-}^l}^{\gamma}(P_{{+}^k{-}^l}^{\alpha})^{*}=P_{\alpha}^{\gamma}.
\end{align*}
\end{proof}

\subsection{Braidings in $\TL$}
We introduce A {\em braiding} $\{c_{\delta,\epsilon}\colon\delta\otimes\epsilon\to\epsilon\otimes\delta\}$ which is a family of isomorphisms in $\TL$.
A definition of the braiding in the diagrammatic category $\TL$ is given by Kuperberg~\cite{Kuperberg94, Kuperberg96}.
In detail about a general theory of braidings in monoidal categories, for example, see~\cite{Turaev91}.

Let us define a description of an $A_2$ web by using a crossing with over/under information introduced in Kuperberg~\cite{Kuperberg96}:
\begin{align*}
&c_{{+},{+}}
=\,\tikz[baseline=-.6ex, scale=.1]{
\draw[->-=.2,->-=.8] (-45:4) -- (135:4);
\draw[->-=.2,->-=.8, white, double=black, double distance=0.4pt, line width=2.4pt] 
(-135:4) -- (45:4);
}\,
=
q^{\frac{1}{3}}
\,\tikz[baseline=-.6ex, scale=.1]{
\draw[->-=.5] (-2,-3) -- (-2,3);
\draw[->-=.5] (2,-3) -- (2,3);
}\,
-
q^{-\frac{1}{6}}
\,\tikz[baseline=-.6ex, scale=.1]{
\draw[->-=.5] (-45:4) -- (0,-1);
\draw[->-=.5] (-135:4) -- (0,-1);
\draw[-<-=.5] (0,-1) -- (0,1);
\draw[-<-=.5] (45:4) -- (0,1);
\draw[-<-=.5] (135:4) -- (0,1);
}\, ,
&c_{{-},{-}}
=\,\tikz[baseline=-.6ex, scale=.1]{
\draw[-<-=.2,-<-=.8] (-45:4) -- (135:4);
\draw[-<-=.2,-<-=.8, white, double=black, double distance=0.4pt, line width=2.4pt] 
(-135:4) -- (45:4);
}\,
=
q^{\frac{1}{3}}
\,\tikz[baseline=-.6ex, scale=.1]{
\draw[-<-=.5] (-2,-3) -- (-2,3);
\draw[-<-=.5] (2,-3) -- (2,3);
}\,
-
q^{-\frac{1}{6}}
\,\tikz[baseline=-.6ex, scale=.1]{
\draw[-<-=.5] (-45:4) -- (0,-1);
\draw[-<-=.5] (-135:4) -- (0,-1);
\draw[->-=.5] (0,-1) -- (0,1);
\draw[->-=.5] (45:4) -- (0,1);
\draw[->-=.5] (135:4) -- (0,1);
}\, ,\\
&c_{{+},{-}}
=\,\tikz[baseline=-.6ex, scale=.1]{
\draw[-<-=.2,-<-=.8] (-45:4) -- (135:4);
\draw[->-=.2,->-=.8, white, double=black, double distance=0.4pt, line width=2.4pt] 
(-135:4) -- (45:4);
}\,
=
q^{-\frac{1}{3}}
\,\tikz[baseline=-.6ex, scale=.1]{
\draw[->-=.5] (-2,3) to[out=south, in=west] (0,1) to[out=east, in=south] (2,3);
\draw[->-=.5] (-2,-3) to[out=north, in=west] (0,-1) to[out=east, in=north] (2,-3);
}\,
-
q^{\frac{1}{6}}
\,\tikz[baseline=-.6ex, scale=.1, rotate=-90]{
\draw[->-=.5] (-45:4) -- (0,-1);
\draw[->-=.5] (-135:4) -- (0,-1);
\draw[-<-=.5] (0,-1) -- (0,1);
\draw[-<-=.5] (45:4) -- (0,1);
\draw[-<-=.5] (135:4) -- (0,1);
}\, ,
&c_{{-},{+}}
=\,\tikz[baseline=-.6ex, scale=.1]{
\draw[->-=.2,->-=.8] (-45:4) -- (135:4);
\draw[-<-=.2,-<-=.8, white, double=black, double distance=0.4pt, line width=2.4pt] 
(-135:4) -- (45:4);
}\,
=
q^{-\frac{1}{3}}
\,\tikz[baseline=-.6ex, scale=.1]{
\draw[-<-=.5] (-2,3) to[out=south, in=west] (0,1) to[out=east, in=south] (2,3);
\draw[-<-=.5] (-2,-3) to[out=north, in=west] (0,-1) to[out=east, in=north] (2,-3);
}\,
-
q^{\frac{1}{6}}
\,\tikz[baseline=-.6ex, scale=.1, rotate=90]{
\draw[->-=.5] (-45:4) -- (0,-1);
\draw[->-=.5] (-135:4) -- (0,-1);
\draw[-<-=.5] (0,-1) -- (0,1);
\draw[-<-=.5] (45:4) -- (0,1);
\draw[-<-=.5] (135:4) -- (0,1);
}\, .
\end{align*}
We also describe their inverses as
\begin{align*}
&c_{{+},{+}}^{-1}
=
\,\tikz[baseline=-.6ex, scale=.1, yscale=-1]{
\draw[-<-=.2,-<-=.8] (-45:4) -- (135:4);
\draw[-<-=.2,-<-=.8, white, double=black, double distance=0.4pt, line width=2.4pt] 
(-135:4) -- (45:4);
}\,
=c_{{+},{+}}^{*}, 
&c_{{-},{-}}^{-1}
=
\,\tikz[baseline=-.6ex, scale=.1, yscale=-1]{
\draw[->-=.2,->-=.8] (-45:4) -- (135:4);
\draw[->-=.2,->-=.8, white, double=black, double distance=0.4pt, line width=2.4pt] 
(-135:4) -- (45:4);
}\,
=c_{{-},{-}}^{*}, \\
&c_{{+},{-}}^{-1}
=
\,\tikz[baseline=-.6ex, scale=.1, yscale=-1]{
\draw[->-=.2,->-=.8] (-45:4) -- (135:4);
\draw[-<-=.2,-<-=.8, white, double=black, double distance=0.4pt, line width=2.4pt] 
(-135:4) -- (45:4);
}\,
=c_{{+},{-}}^{*}, 
&c_{{-},{+}}^{-1}
=
\,\tikz[baseline=-.6ex, scale=.1, yscale=-1]{
\draw[-<-=.2,-<-=.8] (-45:4) -- (135:4);
\draw[->-=.2,->-=.8, white, double=black, double distance=0.4pt, line width=2.4pt] 
(-135:4) -- (45:4);
}\,
=c_{{-},{+}}^{*}.
\end{align*}
By the above description, 
we can consider $A_2$ webs with over/under crossings.
These $A_2$ webs satisfies the {\em Reidemeister moves} (R1)--(R4) for framed tangled trivalent graphs, 
that is, we can confirm the following local moves of $A_2$ webs:
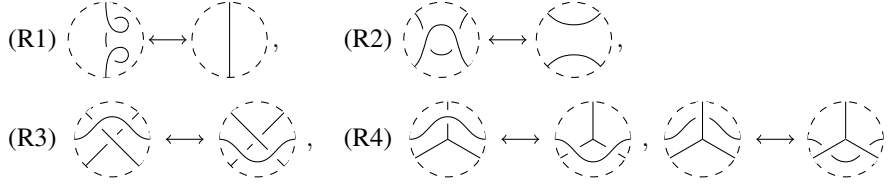
\begin{align*}
&\text{(R1)}~
\,\tikz[baseline=-.6ex, scale=.1]{
\draw [thin, dashed] (0,0) circle [radius=5];
\draw (3,-2)
to[out=south, in=east] (2,-3)
to[out=west, in=south] (0,0)
to[out=north, in=west] (2,3)
to[out=east, in=north] (3,2);
\draw[white, double=black, double distance=0.4pt, line width=2.4pt] (0,-5) 
to[out=north, in=west] (2,-1)
to[out=east, in=north] (3,-2);
\draw[white, double=black, double distance=0.4pt, line width=2.4pt] (3,2)
to[out=south, in=east] (2,1)
to[out=west, in=south] (0,5);
}\,
\tikz[baseline=-.6ex, scale=.5]{
\draw [<->, xshift=1.5cm] (1,0)--(2,0);
}\,
\tikz[baseline=-.6ex, scale=.1, xshift=3cm]{
\draw[thin, dashed] (0,0) circle [radius=5];
\draw (90:5) to (-90:5);
}\, ,
&&\text{(R2)}~
\,\tikz[baseline=-.6ex, scale=.1]{
\draw[thin, dashed] (0,0) circle [radius=5];
\draw (135:5) to[out=south east, in=west] (0,-2) to[out=east, in=south west](45:5);
\draw[white, double=black, double distance=0.4pt, line width=2.4pt]
(-135:5) to[out=north east, in=west] (0,2) to[out=east, in=north west] (-45:5);
}\,
\,\tikz[baseline=-.6ex, scale=.5]{
\draw [<->, xshift=1.5cm] (1,0)--(2,0);
}\,
\,\tikz[baseline=-.6ex, scale=.1, xshift=3cm]{
\draw[thin, dashed] (0,0) circle [radius=5];
\draw (135:5) to[out=south east, in=west](0,2) to[out=east, in=south west](45:5);
\draw (-135:5) to[out=north east, in=west](0,-2) to[out=east, in=north west] (-45:5);
}\, ,\\
&\text{(R3)}~
\,\tikz[baseline=-.6ex, scale=.1]{
\draw [thin, dashed] (0,0) circle [radius=5];
\draw (-135:5) -- (45:5);
\draw[white, double=black, double distance=0.4pt, line width=2.4pt] (135:5) -- (-45:5);
\draw[white, double=black, double distance=0.4pt, line width=2.4pt] 
(180:5) to[out=east, in=west](0,3) to[out=east, in=west] (-0:5);
}\,
\,\tikz[baseline=-.6ex, scale=.5]{
\draw[<->, xshift=1.5cm] (1,0)--(2,0);
}\,
\,\tikz[baseline=-.6ex, scale=.1, xshift=3cm]{
\draw [thin, dashed] (0,0) circle [radius=5];
\draw (-135:5) -- (45:5);
\draw [white, double=black, double distance=0.4pt, line width=2.4pt] (135:5) -- (-45:5);
\draw[white, double=black, double distance=0.4pt, line width=2.4pt] 
(180:5) to[out=east, in=west] (0,-3) to[out=east, in=west] (-0:5);
}\, ,
&&\text{(R4)}~
\,\tikz[baseline=-.6ex, scale=.1]{
\draw [thin, dashed] (0,0) circle [radius=5];
\draw (0:0) -- (90:5); 
\draw (0:0) -- (210:5); 
\draw (0:0) -- (-30:5);
\draw[white, double=black, double distance=0.4pt, line width=2.4pt]
(180:5) to[out=east, in=west] (0,3) to[out=east, in=west] (0:5);
}\,
\,\tikz[baseline=-.6ex, scale=.5]{
\draw[<->, xshift=1.5cm] (1,0)--(2,0);
}\,
\,\tikz[baseline=-.6ex, scale=.1]{
\draw [thin, dashed] (0,0) circle [radius=5];
\draw (0:0) -- (90:5); 
\draw (0:0) -- (210:5); 
\draw (0:0) -- (-30:5);
\draw[white, double=black, double distance=0.4pt, line width=2.4pt]
(180:5) to[out=east, in=west] (0,-3) to[out=east, in=west](0:5);
}\, , 
\,\tikz[baseline=-.6ex, scale=.1]{
\draw [thin, dashed] (0,0) circle [radius=5];
\draw[white, double=black, double distance=0.4pt, line width=2.4pt]
(180:5) to[out=east, in=west] (0,3) to[out=east, in=west] (0:5);
\draw[white, double=black, double distance=0.4pt, line width=2.4pt] (0:0) -- (90:5); 
\draw (0:0) -- (210:5); 
\draw (0:0) -- (-30:5);
}\,
\,\tikz[baseline=-.6ex, scale=.5]{
\draw[<->, xshift=1.5cm] (1,0)--(2,0);
}\,
\,\tikz[baseline=-.6ex, scale=.1]{
\draw [thin, dashed] (0,0) circle [radius=5];
\draw[white, double=black, double distance=0.4pt, ultra thick]
(180:5) to[out=east, in=west] (0,-3) to[out=east, in=west](0:5);
\draw (0:0) -- (90:5); 
\draw[white, double=black, double distance=0.4pt, line width=2.4pt] (0:0) -- (210:5); 
\draw[white, double=black, double distance=0.4pt, line width=2.4pt] (0:0) -- (-30:5);
}\, .
\end{align*}
For any objects $\delta$ and $\epsilon$ in $\TL$, 
we define 
\[
c_{{\delta},{\epsilon}}
=
\,\tikz[baseline=-.6ex, scale=.1]{
\draw (-45:4) -- (135:4);
\draw[white, double=black, double distance=0.4pt, line width=2.4pt] 
(-135:4) -- (45:4);
\node at (135:4)[above]{$\scriptstyle{\epsilon}$};
\node at (45:4)[above]{$\scriptstyle{\delta}$};
\node at (-135:4)[below]{$\scriptstyle{\bar{\delta}}$};
\node at (-45:4)[below]{$\scriptstyle{\bar{\epsilon}}$};
}\,
\in\operatorname{\TL}(\delta\epsilon,\epsilon\delta), 
\quad
c_{{\delta},{\epsilon}}^{-1}
=
\,\tikz[baseline=-.6ex, scale=.1]{
\draw (-135:4) -- (45:4);
\draw[white, double=black, double distance=0.4pt, line width=2.4pt] (-45:4) -- (135:4);
\node at (135:4)[above]{$\scriptstyle{\delta}$};
\node at (45:4)[above]{$\scriptstyle{\epsilon}$};
\node at (-135:4)[below]{$\scriptstyle{\bar{\epsilon}}$};
\node at (-45:4)[below]{$\scriptstyle{\bar{\delta}}$};
}\,
\in\operatorname{\TL}(\epsilon\delta,\delta\epsilon),
\]
where an edge labeled by $\delta$ (resp.~$\epsilon$) mean an embedding of $\mathbf{1}_{\delta}$ (resp.~$\mathbf{1}_{\epsilon}$) along it with the same over/under information at every crossings. 
The invariance under the Reidemeister moves (R1)--(R4) provides the invariance under (R1)--(R3) for any labeled edges. 
By the same reason, we can slide $A_2$ clasps across over/under crossings, namely, 
\begin{align}
&c_{\delta,\epsilon}(\mathbf{1}_{\delta}\otimes P_{\epsilon}^{\epsilon})
=(P_{\epsilon}^{\epsilon}\otimes\mathbf{1}_{\delta})c_{\delta,\epsilon}, 
&&c_{\delta,\epsilon}(P_{\delta}^{\delta}\otimes \mathbf{1}_{\epsilon})
=(\mathbf{1}_{\epsilon}\otimes P_{\delta}^{\delta})c_{\delta,\epsilon},\label{A2slide}\\
&c_{\delta,\epsilon}^{-1}(\mathbf{1}_{\epsilon}\otimes P_{\delta}^{\delta})
=(P_{\delta}^{\delta}\otimes\mathbf{1}_{\epsilon})c_{\delta,\epsilon}^{-1}, 
&&c_{\delta,\epsilon}^{-1}(P_{\epsilon}^{\epsilon}\otimes \mathbf{1}_{\delta})
=(\mathbf{1}_{\delta}\otimes P_{\epsilon}^{\epsilon})c_{\delta,\epsilon}^{-1}.\notag
\end{align}

\section{A categorification of $\CK{k,l}(x,y)$}
Let us briefly recall the definition of the {\em Karoubi envelope} and the {\em split Grothendieck group}.

\begin{DEF}[The Karoubi envelope]
The {\em Karoubi envelope} of a category $\mathcal{A}$, denoted by $\operatorname{Kar}(\mathcal{A})$, is defined as follows:
\begin{itemize}
\item Objects in $\operatorname{Kar}(\mathcal{A})$ is pairs $(X,f)$ of objects $X$ in $\mathcal{A}$ and idempotents $f\in\operatorname{Hom}_{\mathcal{A}}(X,X)$.
\item Morphisms in $\operatorname{Hom}_{\operatorname{Kar}(\mathcal{A})}((X,f),(Y,g))$ is morphisms $\phi\in\operatorname{Hom}_{\mathcal{A}}(X,Y)$ satisfying $g\circ\phi\circ f=\phi$ for any objexts $(X,f)$ and $(Y,g)$.
\end{itemize}
\end{DEF}
We remark that the identity in $\operatorname{Hom}_{\operatorname{Kar}(\mathcal{A})}((X,f),(X,f))$ is given by $f$.
If $\mathcal{A}$ is monoidal, then the Karoubi envelope $\operatorname{Kar}(\mathcal{A})$ inherits a tensor product with $(X,f)\otimes(Y,g)=(X\otimes Y,f\otimes g)$.

\begin{DEF}[The split Grothendieck group]
The {\em split Grothendieck group} of an additive category $\mathcal{A}$ with $\oplus$ is the abelian group $K_0(\mathcal{A})$ generated by isomorphism classes $\langle X \rangle$ of objects $X$ in $\mathcal{A}$ modulo the relations $\langle X_1\oplus X_2\rangle=\langle X_1\rangle +\langle X_2\rangle$. 
If $\mathcal{A}$ is monoidal, that is $\mathcal{A}$ equipped with a tensor $\otimes$ and the identity object $e$, then $K_0(\mathcal{A})$ inherits a ring structure with the unit $\langle e\rangle$ and the multiplication $\langle X\otimes Y\rangle=\langle X\rangle\cdot\langle Y\rangle$.
\end{DEF}

In this section, 
we consider the pairs $\{(\epsilon, P_{\epsilon}^{\epsilon})\}$ in the Karoubi envelope $\Kar$ of the $A_2$ spider and its split Grothendieck group.
Firstly, 
it is easy to see that we can take a standard representatives for the isomorphism class of $\{(\epsilon, P_{\epsilon}^{\epsilon})\}$.
\begin{LEM}
For any $\epsilon\in\TL^{(k,l)}$,
$(\epsilon, P_{\epsilon}^{\epsilon})\cong ({+}^k{-}^l, P_{{+}^k{-}^l}^{{+}^k{-}^l})$ in $\Kar$.
\end{LEM}
\begin{proof}
$P_{\epsilon}^{{+}^k{-}^l}$ and $P_{{+}^k{-}^l}^{\epsilon}$ give an isomorphism between $(\epsilon, P_{\epsilon}^{\epsilon})$ and $({+}^k{-}^l, P_{{+}^k{-}^l}^{{+}^k{-}^l})$.
In fact, 
Proposition~\ref{generalprop} shows
$P_{{+}^k{-}^l}^{{+}^k{-}^l}P_{\epsilon}^{{+}^k{-}^l}P_{\epsilon}^{\epsilon}=P_{\epsilon}^{{+}^k{-}^l}$ 
and 
$P_{\epsilon}^{\epsilon}P_{{+}^k{-}^l}^{\epsilon}P_{{+}^k{-}^l}^{{+}^k{-}^l}=P_{{+}^k{-}^l}^{\epsilon}$.
Thus, $P_{\epsilon}^{{+}^k{-}^l}$ and $P_{{+}^k{-}^l}^{\epsilon}$ are mophisms in $\Kar$.
By the same reason, $P_{{+}^k{-}^l}^{\epsilon}P_{\epsilon}^{{+}^k{-}^l}=P_{\epsilon}^{\epsilon}$ and $P_{\epsilon}^{{+}^k{-}^l}P_{{+}^k{-}^l}^{\epsilon}=P_{{+}^k{-}^l}^{{+}^k{-}^l}$.
\end{proof}

The multiplication in $K_0(\Kar)$ is commutative because of a property (\ref{A2slide}) of a braiding.
\begin{LEM}
For any objects $\epsilon$ and $\delta$ in $\TL$, $P_{\epsilon}^{\epsilon}\otimes P_{\delta}^{\delta}\cong P_{\delta}^{\delta}\otimes P_{\epsilon}^{\epsilon}$ by $c_{\epsilon,\delta}$.
\end{LEM}

Let us denote objects $({+}^k{-}^l, P_{{+}^k{-}^l}^{{+}^k{-}^l})$ in $\Kar$ by $P_{(k,l)}$ for $k,l\geq 1$. 
We denote the pair $(\epsilon_0, D_{\emptyset})$ of the empty sign and the empty web by $P_{(0,0)}$. 
The split Grothendieck group $K_0(\Kar)$ is a ring with with the unit $1=\langle P_{(0,0)}\rangle$. 
We denote $\langle ({+},\mathbf{1}_{+})\rangle$ and $\langle ({-},\mathbf{1}_{-})\rangle$ by $X$ and $Y$ in $K_0(\Kar)$, respectively. 

We will show the isomorphism classes $\langle P_{(k,l)}\rangle$ satisfy the recursive formula of the $A_2$ Chebyshev polynomials in $K_0(\Kar)$:
\begin{align}
\langle P_{(1,1)}\rangle&=XY-1,\label{CK1}\\
\langle P_{(k+1,0)}\rangle&=X\langle P_{(k,0)}\rangle-\langle P_{(k-1,1)}\rangle\ \text{ for }\ k\geq 1,\label{CK2}\\
\langle P_{(k+1,l)}\rangle&=X\langle P_{(k,l)}\rangle-\langle P_{(k-1,l+1)}\rangle-\langle P_{(k,l-1)}\rangle\ \text{ for }\ k,l\geq 1,\label{CK3}\\
\langle P_{(0,l+1)}\rangle&=Y\langle P_{(0,l)}\rangle-\langle P_{(1,l-1)}\rangle\ \text{ for }\ l\geq 1,\label{CK4}\\
\langle P_{(k,l+1)}\rangle&=Y\langle P_{(k,l)}\rangle-\langle P_{(k+1,l-1)}\rangle-\langle P_{(k-1,l)}\rangle\ \text{ for }\ k,l\geq 1.\label{CK5}
\end{align}
We only have to prove (\ref{CK1})--(\ref{CK3}) because of $(P_{{+}^k{-}^l}^{{+}^k{-}^l})^{*}\cong P_{{-}^k{+}^l}^{{-}^k{+}^l}$.

We use the following well-known fact about an additive category, see~\cite{MacLane98, KashiwaraShapira06}, for example.
\begin{LEM}\label{projincl}
Let $\mathcal{A}$ be an additive category and $X_1,X_2,Y$ objects in $\mathcal{A}$.
$Y\cong X_1\oplus X_2$ if and only if there exists morphisms $p_1\colon Y\to X_1$, $p_2\colon Y\to X_2$, $\iota_1\colon X_1\to Y$, and $\iota_2\colon X_2\to Y$ satisfying the following conditions:
\[
\iota_1\circ p_1+\iota_2\circ p_2=\mathbf{1}_{Y},
p_1\circ\iota_2=p_2\circ\iota_1=0,
p_1\circ\iota_1=\mathbf{1}_{X_1}, 
p_2\circ\iota_2=\mathbf{1}_{X_2}.
\]
\end{LEM}

\begin{PROP}\label{pfCK1}
$\langle P_{(1,1)}\rangle=XY-1$
\end{PROP}
\begin{proof}
By Definition~{\ref{OYdef}}, 
\[
P_{{+}{-}}^{{+}{-}}+\frac{1}{[3]}b^{{+}{-}}d^{{+}{-}}=\mathbf{1}_{+}\otimes\mathbf{1}_{-}.
\]
Let us prove $({+}{-},P_{{+}{-}}^{{+}{-}}+\frac{1}{[3]} b^{{+}{-}}d_{{+}{-}})\cong P_{(1,1)} \oplus P_{(0,0)}$.
It is only necessary to construct projections $p_1\colon ({+}{-},P_{{+}{-}}^{{+}{-}}+\frac{1}{[3]} b^{{+}{-}}d_{{+}{-}})\to P_{(1,1)}$ and $p_2\colon ({+}{-},P_{{+}{-}}^{{+}{-}}+\frac{1}{[3]} b^{{+}{-}}d_{{+}{-}})\to P_{(0,0)}$, and inclusions $\iota_1\colon P_{(1,1)}\to({+}{-},P_{{+}{-}}^{{+}{-}}+\frac{1}{[3]} b^{{+}{-}}d_{{+}{-}})$ and $\iota_2\colon P_{(0,0)}\to({+}{-},P_{{+}{-}}^{{+}{-}}+\frac{1}{[3]} b^{{+}{-}}d_{{+}{-}})$ satisfying the conditions in Lemma~\ref{projincl}.
These projections are given by $p_1=P_{{+}{-}}^{{+}{-}}$ and $p_2=\frac{1}{[3]}d_{{+}{-}}$ because 
\begin{align*}
P_{{+}{-}}^{{+}{-}}(P_{{+}{-}}^{{+}{-}}+\frac{1}{[3]}b^{{+}{-}}d_{{+}{-}})
=(P_{{+}{-}}^{{+}{-}})^2+\frac{1}{[3]}(P_{{+}{-}}^{{+}{-}}b^{{+}{-}})d_{{+}{-}}
=P_{{+}{-}}^{{+}{-}}
\end{align*}
by Proposition~\ref{doubleproj}, and
\begin{align*}
\frac{1}{[3]}d_{{+}{-}}(P_{{+}{-}}^{{+}{-}}+\frac{1}{[3]}b^{{+}{-}}d_{{+}{-}})
=\frac{1}{[3]}d_{{+}{-}}P_{{+}{-}}^{{+}{-}}+\frac{1}{[3]^2}d_{{+}{-}}b^{{+}{-}}d_{{+}{-}}
=\frac{1}{[3]}d_{{+}{-}}
\end{align*}
by Proposition~\ref{doubleproj} and 
$d_{{+}{-}}b^{{+}{-}}
=\,\tikz[baseline=-.6ex, scale=.6]{
\draw[-<-=.5] (0,0) circle [radius=.3];
}\,
=[3]$.
These inclusions are given by $\iota_1=P_{{+}{-}}^{{+}{-}}$ and $\iota_2=b^{{+}{-}}$ because
\begin{align*}
(P_{{+}{-}}^{{+}{-}}+\frac{1}{[3]}b^{{+}{-}}d_{{+}{-}})P_{{+}{-}}^{{+}{-}}
=(P_{{+}{-}}^{{+}{-}})^2+\frac{1}{[3]}(b^{{+}{-}})d_{{+}{-}}P_{{+}{-}}^{{+}{-}}
=P_{{+}{-}}^{{+}{-}}
\end{align*}
and
\begin{align*}
(P_{{+}{-}}^{{+}{-}}+\frac{1}{[3]}b^{{+}{-}}d_{{+}{-}})b^{{+}{-}}
=P_{{+}{-}}^{{+}{-}}b^{{+}{-}}+\frac{1}{[3]}b^{{+}{-}}d_{{+}{-}}b^{{+}{-}}
=b^{{+}{-}}.
\end{align*}
In fact, it is easily to see that
\begin{align*}
&\iota_1\circ p_2=\iota_2\circ p_1=0, 
&&\iota_1\circ p_1=P_{{+}{-}}^{{+}{-}}, 
&&\iota_2\circ p_2=\frac{1}{[3]}b^{{+}{-}}d_{{+}{-}},\\
&p_1\circ\iota_2=p_2\circ\iota_1=0, 
&&p_1\circ\iota_1=P_{{+}{-}}^{{+}{-}}, 
&&p_2\circ\iota_2=D_{\emptyset}.
\end{align*}
Thus, we obtain $P_{(1,1)} \oplus P_{(0,0)}\cong ({+}{-},P_{{+}{-}}^{{+}{-}}+\frac{1}{[3]} b^{{+}{-}}d_{{+}{-}})$.
In $K_0(\Kar)$,
it is interpreted as $\langle P_{(1,1)}\rangle+1=\langle ({+}{-},P_{{+}{-}}^{{+}{-}}+\frac{1}{[3]} b^{{+}{-}}d_{{+}{-}})\rangle=XY$.
\end{proof}

\begin{PROP}\label{pfCK2}
For any positive integer $k$,
$\langle P_{(k+1,0)}\rangle=X\langle P_{(k,0)}\rangle-\langle P_{(k-1,1)}\rangle$.
\end{PROP}
\begin{proof}
By Definition~\ref{singledef}, 
\begin{align}\label{eqCK2}
 P_{{+}^{k+1}}^{{+}^{k+1}}
+\frac{[k]}{[k+1]}
(P_{{+}^k}^{{+}^k}\otimes\mathbf{1}_{+})
(\mathbf{1}_{{+}^{k-1}}\otimes t_{{-}}^{{+}{+}}t_{{+}{+}}^{-})
(P_{{+}^k}^{{+}^k}\otimes\mathbf{1}_{+})=P_{{+}^k}^{{+}^k}\otimes\mathbf{1}_{+}
\end{align}
Let us denote the LHS of (\ref{eqCK2}) by $f$.
We show that the pair $({+}^{k+1},f)$ is isomorphic to the direct sum $P_{(k+1,0)}\oplus P_{(k-1,1)}$ in $\Kar$. 
We only have to construct projections $p_1\colon ({+}^{k+1},f)\to P_{(k+1,0)}$ and $p_2\colon ({+}^{k+1},f)\to P_{(k-1,1)}$, and inclusions $\iota_1\colon P_{(k+1,0)}\to ({+}^{k+1},f)$ and $\iota_2\colon P_{(k-1,1)}\to ({+}^{k+1},f)$ satisfying (\ref{projincl}).
We confirm that 
\[
p_1=P_{{+}^{k+1}}^{{+}^{k+1}}
\quad\text{and}\quad 
p_2=\frac{[k]}{[k+1]}P_{{+}^{k-1}{-}}^{{+}^{k-1}{-}}(\mathbf{1}_{{+}^{k-1}}\otimes t_{{+}{+}}^{-})(P_{{+}^{k}}^{{+}^{k}}\otimes\mathbf{1}_{+}),
\]
provide projections and
\[
\iota_1=P_{{+}^{k+1}}^{{+}^{k+1}}
\quad\text{and}\quad
\iota_2=(P_{{+}^{k}}^{{+}^{k}}\otimes\mathbf{1}_{+})(\mathbf{1}_{{+}^{k-1}}\otimes t_{{-}}^{{+}{+}})P_{{+}^{k-1}{-}}^{{+}^{k-1}{-}}
\]
 inclusions.
It is easy to see that $p_1$ and $\iota_1$ is a morphism between $({+}^{k+1},f)$ and $P_{(k+1,0)}$. 
We show that $p_2$ is a morphism in $\Kar$, that is, 
\begin{align*}
P_{{+}^{k-1}{-}}^{{+}^{k-1}{-}}P_{{+}^{k-1}{-}}^{{+}^{k-1}{-}}(\mathbf{1}_{{+}^{k-1}}\otimes t_{{+}{+}}^{-})(P_{{+}^{k}}^{{+}^{k}}\otimes\mathbf{1}_{+})f
=P_{{+}^{k-1}{-}}^{{+}^{k-1}{-}}(\mathbf{1}_{{+}^{k-1}}\otimes t_{{+}{+}}^{-})(P_{{+}^{k}}^{{+}^{k}}\otimes\mathbf{1}_{+}).
\end{align*}
By Proposition~\ref{doubleproj}, the LHS of the above equation is 
\[
\frac{[k]}{[k+1]}P_{{+}^{k-1}{-}}^{{+}^{k-1}{-}}(\mathbf{1}_{{+}^{k-1}}\otimes t_{{+}{+}}^{-})(P_{{+}^{k}}^{{+}^{k}}\otimes\mathbf{1}_{+})
(\mathbf{1}_{{+}^{k-1}}\otimes t_{{-}}^{{+}{+}}t_{{+}{+}}^{-})
(P_{{+}^k}^{{+}^k}\otimes\mathbf{1}_{+}).
\]
We diagrammatically calculate it using Definition~\ref{singledef}:
\begin{align}
\frac{[k]}{[k+1]}
\,\tikz[baseline=-.6ex, scale=.1, yscale=.8, yshift=-2cm]{
\draw[->-=.5] (0,9) -- (0,11);
\draw[-<-=.5] (4,9) -- (4,11);
\draw[->-=.5] (0,4) -- (0,8);
\draw[-<-=.5] (4,6) -- (4,8);
\draw[->-=.5] (2,4) to[out=north, in=west] (4,6);
\draw[->-=.5] (6,4) to[out=north, in=east] (4,6);
\draw (6,3) -- (6,4);
\draw[fill=white] (-2,8) rectangle (7,9);
\node at (0,11) [left]{${\scriptstyle k-1}$};
\node at (4,11) [right]{${\scriptstyle 1}$};
\draw[->-=.5] (1,-6) -- (1,-3);
\draw[->-=.5] (0,-3) -- (0,3);
\draw[-<-=.5] (4,-1) -- (4,1);
\draw[-<-=.7] (2,3) to[out=south, in=west] (4,1);
\draw[-<-=.5] (6,3) to[out=south, in=east] (4,1);
\draw[->-=.7, yscale=-1] (2,3) to[out=south, in=west] (4,1);
\draw[->-=.5, yscale=-1] (6,6) -- (6,3) to[out=south, in=east] (4,1);
\draw[fill=white] (-2,3) rectangle (3,4);
\draw[fill=white, yshift=-3cm] (-2,-.5) rectangle (3,.5);
\node at (0,6) [left]{${\scriptstyle k-1}$};
\node at (1,-6) [left]{${\scriptstyle k}$};
\node at (0,0) [left]{${\scriptstyle k-1}$};
}\,
&=
\frac{[k]}{[k+1]}
\Big(
\,\tikz[baseline=-.6ex, scale=.1, yscale=.8, yshift=-2cm]{
\draw[->-=.5] (0,9) -- (0,11);
\draw[-<-=.5] (4,9) -- (4,11);
\draw[-<-=.5] (4,5) -- (4,8);
\draw[fill=white] (-2,8) rectangle (7,9);
\node at (4,11) [right]{${\scriptstyle 1}$};
\node at (0,11) [left]{${\scriptstyle k-1}$};
\draw[->-=.5] (1,-6) -- (1,-3);
\draw[->-=.5] (0,-3) -- (0,8);
\draw[-<-=.5] (4,-1) -- (4,1);
\draw[->-=.5] (4,1) to[out=west, in=west] (4,5);
\draw[->-=.5] (4,1) to[out=east, in=east] (4,5);
\draw[->-=.7, yscale=-1] (2,3) to[out=south, in=west] (4,1);
\draw[->-=.5, yscale=-1] (6,6) -- (6,3) to[out=south, in=east] (4,1);
\draw[fill=white, yshift=-3cm] (-2,-.5) rectangle (3,.5);
\node at (1,-6) [left]{${\scriptstyle k}$};
\node at (0,3) [left]{${\scriptstyle k-1}$};
}\,
-\frac{[k-1]}{[k]}
\,\tikz[baseline=-.6ex, scale=.1, yscale=.8, yshift=-2cm]{
\draw[->-=.5] (0,9) -- (0,11);
\draw[-<-=.5] (4,9) -- (4,11);
\draw[-<-=.5] (4,5) -- (4,8);
\draw[fill=white] (-2,8) rectangle (7,9);
\node at (4,11) [right]{${\scriptstyle 1}$};
\node at (0,11) [left]{${\scriptstyle k-1}$};
%%%%% left squair %%%%%
\draw[->-=.5] (1,-6) -- (1,-3);
\draw[->-=.5] (-1,-3) -- (-1,8);
\draw[-<-=.5] (4,-1) -- (4,1);
\draw[-<-=.5] (1,1) -- (4,1);
\draw[-<-=.5] (1,1) -- (1,5);
\draw[->-=.5] (1,5) -- (4,5);
\draw[->-=.5] (1,-3) -- (1,1);
\draw[->-=.5] (1,5) -- (1,8);
%%%%%%%%%%%%%%%%%%%%%%%
\draw[->-=.5] (4,1) to[out=east, in=east] (4,5);
\draw[->-=.7, yscale=-1] (2,3) to[out=south, in=west] (4,1);
\draw[->-=.5, yscale=-1] (6,6) -- (6,3) to[out=south, in=east] (4,1);
\draw[fill=white, yshift=-3cm] (-2,-.5) rectangle (3,.5);
\node at (1,-6) [left]{${\scriptstyle k}$};
\node at (-1,3) [left]{${\scriptstyle k-2}$};
}\,
\Big)\label{p2CK2}\\
&=
\frac{[k]}{[k+1]}
\Big([2]-\frac{[k-1]}{[k]}\Big)
\,\tikz[baseline=-.6ex, scale=.1, yshift=-1cm]{
\draw[->-=.5] (0,5) -- (0,7);
\draw[-<-=.5] (4,5) -- (4,7);
\draw[fill=white] (-2,4) rectangle (7,5);
\node at (4,7) [right]{${\scriptstyle 1}$};
\node at (0,7) [left]{${\scriptstyle k-1}$};
\draw[->-=.5] (1,-6) -- (1,-3);
\draw[->-=.5] (0,-3) -- (0,4);
\draw[-<-=.5] (4,-1) -- (4,4);
\draw[->-=.7, yscale=-1] (2,3) to[out=south, in=west] (4,1);
\draw[->-=.5, yscale=-1] (6,6) -- (6,3) to[out=south, in=east] (4,1);
\draw[fill=white, yshift=-3cm] (-2,-.5) rectangle (3,.5);
\node at (1,-6) [left]{${\scriptstyle k}$};
\node at (0,1) [left]{${\scriptstyle k-1}$};
}\,\notag\\
&=
\,\tikz[baseline=-.6ex, scale=.1, yshift=-1cm]{
\draw[->-=.5] (0,5) -- (0,7);
\draw[-<-=.5] (4,5) -- (4,7);
\draw[fill=white] (-2,4) rectangle (7,5);
\node at (4,7) [right]{${\scriptstyle 1}$};
\node at (0,7) [left]{${\scriptstyle k-1}$};
\draw[->-=.5] (1,-6) -- (1,-3);
\draw[->-=.5] (0,-3) -- (0,4);
\draw[-<-=.5] (4,-1) -- (4,4);
\draw[->-=.7, yscale=-1] (2,3) to[out=south, in=west] (4,1);
\draw[->-=.5, yscale=-1] (6,6) -- (6,3) to[out=south, in=east] (4,1);
\draw[fill=white, yshift=-3cm] (-2,-.5) rectangle (3,.5);
\node at (1,-6) [left]{${\scriptstyle k}$};
\node at (0,1) [left]{${\scriptstyle k-1}$};
}\,
=P_{{+}^{k-1}{-}}^{{+}^{k-1}{-}}(\mathbf{1}_{{+}^{k-1}}\otimes t_{{+}{+}}^{-})(P_{{+}^{k}}^{{+}^{k}}\otimes\mathbf{1}_{+}).\notag
\end{align}
The third equation uses $[2][k]=[k+1]+[k-1]$.
To prove $\iota_2$ is a morphism in $\TL$, 
we have to compute $f(P_{{+}^{k}}^{{+}^{k}}\otimes\mathbf{1}_{+})(\mathbf{1}_{{+}^{k-1}}\otimes t_{{-}}^{{+}{+}})P_{{+}^{k-1}{-}}^{{+}^{k-1}{-}}P_{{+}^{k-1}{-}}^{{+}^{k-1}{-}}$. 
By Proposition~\ref{doubleproj}, it reduce to
\[
\frac{[k]}{[k+1]}
(P_{{+}^k}^{{+}^k}\otimes\mathbf{1}_{+})
(\mathbf{1}_{{+}^{k-1}}\otimes t_{{-}}^{{+}{+}}t_{{+}{+}}^{-})
(P_{{+}^{k}}^{{+}^{k}}\otimes\mathbf{1}_{+})(\mathbf{1}_{{+}^{k-1}}\otimes t_{{-}}^{{+}{+}})P_{{+}^{k-1}{-}}^{{+}^{k-1}{-}}.
\]
We can calculate it by turning the diagrams in (\ref{p2CK2}) upside down.
Moreover, it can be confirmed that 
$
\iota_2\circ p_2=\frac{[k]}{[k+1]}
(P_{{+}^k}^{{+}^k}\otimes\mathbf{1}_{+})
(\mathbf{1}_{{+}^{k-1}}\otimes t_{{-}}^{{+}{+}}t_{{+}{+}}^{-})
(P_{{+}^k}^{{+}^k}\otimes\mathbf{1}_{+})=P_{{+}^k}^{{+}^k}\otimes\mathbf{1}_{+}
$
by Definition~\ref{OYdef} and Poroposition~\ref{doubleproj} and 
$p_2\circ\iota_2=P_{{+}^{k-1}{-}}^{{+}^{k-1}{-}}$ by inserting $P_{{+}^{k-1}{-}}^{{+}^{k-1}{-}}$ into the diagrams of (\ref{p2CK2}).
Therefore, these projections and inclusions satisfy the condition in Lemma~\ref{projincl}  and give $({+}^{k+1},f)\cong P_{(k+1,0)}\oplus P_{(k-1,1)}$ in $\Kar$.
In terms of $K_0(\Kar)$, 
\[
\langle P_{(k,0)}\rangle X=\langle ({+}^{k+1},f)\rangle=\langle P_{(k+1,0)}\rangle +  \langle P_{(k-1,1)}\rangle.
\]
\end{proof}

Let us prepare some lemmata to prove (\ref{CK3}).
We introduce an $A_2$ basis web represented by a rectangle with a diagonal line which consist of $H_{{+}{-}}^{{-}{+}}$ and $H_{{-}{+}}^{{{+}{-}}}$. 
\begin{DEF} 
Let $m$ and $n$ be positive integers.
\begin{enumerate}
\item
$
\,\tikz[baseline=-.6ex, scale=.1]{
\draw[->-=.5] (0,-3) -- (0,-1);
\draw[->-=.5] (0,1) -- (0,3);
\draw[-<-=.8] (-2,0) -- (-4,0) -- (-4,3);
\draw[->-=.8] (2,0) -- (4,0) -- (4,-3);
\draw[fill=white] (-2,-1) rectangle (2,1);
\draw (-2,1) -- (2,-1);
\node at (-4,3) [above]{${\scriptstyle 1}$};
\node at (4,-3) [below]{${\scriptstyle 1}$};
\node at (0,3) [above]{${\scriptstyle n}$};
\node at (0,-3) [below]{${\scriptstyle n}$};
}\,
=
\,\tikz[baseline=-.6ex, scale=.1]{
\draw[->-=.5] (-3,-3) -- (-3,0);
\draw[->-=.5] (-2.5,0) -- (-2.5,3);
\draw[->-=.5] (-2,-3) -- (-2,0);
\draw[->-=.5] (-1.5,0) -- (-1.5,3);
\node at (1,0) [above]{${\scriptstyle \cdots}$};
\node at (.5,0) [below]{${\scriptstyle \cdots}$};
\draw[->-=.5] (3,-3) -- (3,0);
\draw[->-=.5] (3.5,0) -- (3.5,3);
\draw[-<-=.8] (-2,0) -- (-4,0) -- (-4,3);
\draw (-2,0) -- (2,0);
\draw[->-=.8] (2,0) -- (4.5,0) -- (4.5,-3);
}\,
\in \TL({+}^n{-},{-}{+}^n)
$
\item
$
\,\tikz[baseline=-.6ex, scale=.1]{
\draw[->-=.5] (0,-3) -- (0,-1);
\draw[->-=.5] (0,1) -- (0,3);
\draw[-<-=.8] (-2,0) -- (-4,0) -- (-4,3);
\draw[->-=.8] (2,0) -- (4,0) -- (4,-3);
\draw[fill=white] (-2,-1) rectangle (2,1);
\draw (-2,1) -- (2,-1);
\node at (-4,3) [above]{${\scriptstyle m}$};
\node at (4,-3) [below]{${\scriptstyle m}$};
\node at (0,3) [above]{${\scriptstyle n}$};
\node at (0,-3) [below]{${\scriptstyle n}$};
}\,
=
\,\tikz[baseline=-.6ex, scale=.1]{
\draw[->-=.5] (0,-3) -- (0,-1);
\draw[->-=.5] (0,1) -- (0,3);
\draw[-<-=.8] (-2,0) -- (-4,0) -- (-4,3);
\draw[->-=.8] (2,0) -- (4,0) -- (4,-3);
\draw[-<-=.5] (-6,-3) -- (-6,3);
\draw[fill=white] (-2,-1) rectangle (2,1);
\draw (-2,1) -- (2,-1);
\node at (-4,3) [above]{${\scriptstyle 1}$};
\node at (4,-3) [below]{${\scriptstyle 1}$};
\node at (0,3) [above]{${\scriptstyle n}$};
\node at (0,-3) [below]{${\scriptstyle n}$};
\node at (-6,-3) [below]{${\scriptstyle m-1}$};
}\,
\circ
\,\tikz[baseline=-.6ex, scale=.1]{
\draw[->-=.5] (0,-3) -- (0,-1);
\draw[->-=.5] (0,1) -- (0,3);
\draw[-<-=.8] (-2,0) -- (-4,0) -- (-4,3);
\draw[->-=.8] (2,0) -- (4,0) -- (4,-3);
\draw[-<-=.5] (6,-3) -- (6,3);
\draw[fill=white] (-2,-1) rectangle (2,1);
\draw (-2,1) -- (2,-1);
\node at (-4,3) [above]{${\scriptstyle m-1}$};
\node at (4,-3) [below]{${\scriptstyle m-1}$};
\node at (0,3) [above]{${\scriptstyle n}$};
\node at (0,-3) [below]{${\scriptstyle n}$};
\node at (6,3) [above]{${\scriptstyle 1}$};
}\,
\in \TL({+}^n{-}^m,{-}^m{+}^n)
$
\end{enumerate}
\end{DEF}

\begin{LEM}[Kim~{\cite[Proposition~3.1]{Kim07}}]\label{singleexp}
For any positive integer $k$,
\[
\,\tikz[baseline=-.6ex, scale=.1]{
\draw[->-=.2,->-=.8] (0,-3) -- (0,4);
\draw[fill=white] (-3,0) rectangle (3,1);
\node at (0,3) [left]{${\scriptstyle k}$};
}\,
=
\sum_{j=0}^{k-1}(-1)^{j}\frac{[k-j]}{[k]}
\,\tikz[baseline=-.6ex, scale=.1, yshift=-2cm]{
\draw[->-=.2, ->-=.8, rounded corners] (-4,-3) -- (-4,3) -- (2,3) -- (2,8);
\draw[->-=.8] (-1,0) -- (-1,8);
\draw[->-=.5] (4,0) -- (4,8);
\draw[->-=.5] (1,-3) -- (1,0);
\draw[fill=white] (-3,2) rectangle (1,4);
\draw (-3,4) -- (1,2);
\draw[fill=white] (-3,0) rectangle (5,1);
\node at (-1,8) [above]{${\scriptstyle j}$};
\node at (2,8) [above]{${\scriptstyle 1}$};
\node at (3,8) [above right]{${\scriptstyle k-j-1}$};
\node at (1,-3) [below]{${\scriptstyle k-1}$};
\node at (-4,-3) [below]{${\scriptstyle 1}$};
}\,
\]
\end{LEM}

\begin{LEM}\label{lem2}
Let $k$ be a positive integer and 
$
X(k;i)
=
\,\tikz[baseline=-.6ex, scale=.1, yscale=.8]{
\draw[->-=.2, -<-=.8] (-2,-4) -- (-2,9);
\draw[->-=.8, rounded corners] (1,1) -- (1,3) -- (-2,3);
\draw[->-=.2, ->-=.8] (3,1) -- (3,9);
\draw[-<-=.2, -<-=.8] (9,1) -- (9,9);
\draw[-<-=.5, rounded corners] (4,0) -- (4,-1) -- (9,-1) -- (9,1);
\draw[-<-=.5] (2,0) -- (2,-4);
\draw[fill=white] (0,5) rectangle (12,6);
\draw[fill=white] (0,0) rectangle (5,1);
\node at (3,9) [above]{${\scriptstyle k}$};
\node at (9,9) [above]{${\scriptstyle i}$};
\node at (3,3) [right]{${\scriptstyle k}$};
\node at (9,3) [right]{${\scriptstyle i}$};
\node at (2,-4) [below]{${\scriptstyle k+1-i}$};
}\, .
$
Then, $X(k;i)=0$ for $1<i<k+1$ and 
\[
X(k;1)
=\frac{(-1)^{k}}{[k+1]}
\,\tikz[baseline=-.6ex, scale=.1, yscale=.8, yshift=-4cm]{
\draw[->-=.2, -<-=.8] (-2,0) -- (-2,9);
\draw[->-=.8, rounded corners] (1,5) -- (1,3) -- (-2,3);
\draw[->-=.5] (2,6) -- (2,9);
\draw[->-=.5] (4,0) -- (4,5);
\draw[-<-=.5] (5,6) -- (5,9);
\draw[fill=white] (0,5) rectangle (7,6);
\node at (2,9) [above]{${\scriptstyle k}$};
\node at (5,9) [above]{${\scriptstyle 1}$};
\node at (4,0) [below]{${\scriptstyle k}$};
}\, 
=
\frac{(-1)^{k}}{[k+1]}
(\mathbf{1}_{-}\otimes P_{{-}{+}^k{-}^{l-1}}^{{+}^k{-}^{l}})
(t_{{+}}^{{-}{-}}\otimes P_{{+}^k{-}^{l-1}}^{{+}^k{-}^{l-1}}).
\]
\end{LEM}

\begin{proof}
$A_2$ webs appearing in the RHS of Lemma~\ref{singleexp} contains $t_{{-}}^{{+}{+}}$ if $j\geq 1$ in the top side. 
Thus, 
\[
\,\tikz[baseline=-.6ex, scale=.1]{
\draw[->-=.2, -<-=.8] (-2,-4) -- (-2,9);
\draw[->-=.8, rounded corners] (1,1) -- (1,3) -- (-2,3);
\draw[->-=.2, ->-=.8] (3,1) -- (3,9);
\draw[-<-=.2, -<-=.8] (9,1) -- (9,9);
\draw[-<-=.5, rounded corners] (4,0) -- (4,-1) -- (9,-1) -- (9,1);
\draw[-<-=.5] (2,0) -- (2,-4);
\draw[fill=white] (0,5) rectangle (12,6);
\draw[fill=white] (0,0) rectangle (5,1);
\node at (3,9) [above]{${\scriptstyle k}$};
\node at (9,9) [above]{${\scriptstyle i}$};
\node at (3,3) [right]{${\scriptstyle k}$};
\node at (9,3) [right]{${\scriptstyle i}$};
\node at (2,-4) [below]{${\scriptstyle k+1-i}$};
}\,
=
\,\tikz[baseline=-.6ex, scale=.1]{
\draw[->-=.2, -<-=.8] (-2,-4) -- (-2,9);
\draw[->-=.5, rounded corners] (0,-4) -- (0,3) -- (-2,3);
\draw[->-=.2, ->-=.8] (3,1) -- (3,9);
\draw[-<-=.2, -<-=.8] (9,1) -- (9,9);
\draw[-<-=.5, rounded corners] (4,0) -- (4,-1) -- (9,-1) -- (9,1);
\draw[-<-=.5] (2,0) -- (2,-4);
\draw[fill=white] (0,5) rectangle (12,6);
\draw[fill=white] (1,0) rectangle (5,1);
\node at (3,9) [above]{${\scriptstyle k}$};
\node at (9,9) [above]{${\scriptstyle i}$};
\node at (3,3) [right]{${\scriptstyle k}$};
\node at (9,3) [right]{${\scriptstyle i}$};
\node at (2,-4) [right]{${\scriptstyle k-i}$};
}\,
-\frac{[k]}{[k+1]}
\,\tikz[baseline=-.6ex, scale=.1]{
\draw[->-=.2, -<-=.8] (-3,-4) -- (-3,10);
\draw[-<-=.5] (-3,4) -- (-1,4);
\draw[->-=.2] (-1,-4) -- (-1,6);
\draw[->-=.8, rounded corners] (1,1) -- (1,3) -- (-1,3);
\draw[->-=.5] (3,1) -- (3,6);
\draw[->-=.5] (1,6) -- (1,10);
\draw[-<-=.2, -<-=.8] (9,1) -- (9,10);
\draw[-<-=.5, rounded corners] (4,0) -- (4,-1) -- (9,-1) -- (9,1);
\draw[-<-=.5] (2,0) -- (2,-4);
\draw[fill=white] (-2,6) rectangle (12,7);
\draw[fill=white] (0,0) rectangle (5,1);
\node at (1,10) [above]{${\scriptstyle k}$};
\node at (9,10) [above]{${\scriptstyle i}$};
\node at (2,4) [right]{${\scriptstyle k-1}$};
\node at (9,3) [right]{${\scriptstyle i}$};
\node at (2,-4) [right]{${\scriptstyle k-i}$};
}\,
\]
by applying Lemma~\ref{singleexp} to $P_{{+}^k}^{{+}^k}$. 
The $A_2$ web in the second term has $X(k-1;i)$.
We obtain proof by induction on $k$. 
\end{proof}

\begin{PROP}[Kim~{\cite[Theorem~3.3]{Kim07}}]\label{prop5}
For any positive integers $k$ and $l$,

\begin{align*}
\,\tikz[baseline=-.6ex, scale=.1]{
\draw[->-=.5] (2,1) -- (2,4);
\draw[->-=.5] (2,-3) -- (2,0);
\draw[-<-=.5] (5,1) -- (5,4);
\draw[-<-=.5] (5,-3) -- (5,0);
\draw[fill=white] (0,0) rectangle (7,1);
\node at (1,4) [above]{${\scriptstyle k+1}$};
\node at (6,4) [above]{${\scriptstyle l}$};
\node at (1,-3) [below]{${\scriptstyle k+1}$};
\node at (6,-3) [below]{${\scriptstyle l}$};
}\,
=
\,\tikz[baseline=-.6ex, scale=.1]{
\draw[->-=.5] (-1,-3) -- (-1,4);
\draw[->-=.5] (2,1) -- (2,4);
\draw[->-=.5] (2,-3) -- (2,0);
\draw[-<-=.5] (5,1) -- (5,4);
\draw[-<-=.5] (5,-3) -- (5,0);
\draw[fill=white] (0,0) rectangle (7,1);
\node at (-1,4) [above]{${\scriptstyle 1}$};
\node at (2,4) [above]{${\scriptstyle k}$};
\node at (5,4) [above]{${\scriptstyle l}$};
\node at (2,-3) [below]{${\scriptstyle k}$};
\node at (5,-3) [below]{${\scriptstyle l}$};
}\,
-\frac{[k]}{[k+1]}
\,\tikz[baseline=-.6ex, scale=.1, yscale=.7]{
\begin{scope}[yshift=3cm]
\draw[->-=.5] (-2,0) -- (-2,4);
\draw[->-=.5] (2,1) -- (2,4);
\draw[-<-=.5] (7,1) -- (7,4);
\draw[fill=white] (-1,0) rectangle (9,1);
\node at (-2,4) [above]{${\scriptstyle 1}$};
\node at (2,4) [above]{${\scriptstyle k}$};
\node at (7,4) [above]{${\scriptstyle l}$};
\end{scope}
\draw (-2,-3) -- (-2,3);
\draw[->-=.5] (1,-3) -- (1,3);
\draw[-<-=.5] (8,-3) -- (8,3);
\draw[->-=.5] (-2,1) -- (0,1) -- (0,3);
\draw[-<-=.5] (-2,-1) -- (0,-1) -- (0,-3);
\node at (.5,0) [right]{${\scriptstyle k-1}$};
\node at (8,0) [right]{${\scriptstyle l}$};
\begin{scope}[yshift=-3cm, yscale=-1]
\draw[-<-=.5] (-2,0) -- (-2,4);
\draw[-<-=.5] (2,1) -- (2,4);
\draw[->-=.5] (7,1) -- (7,4);
\draw[fill=white] (-1,0) rectangle (9,1);
\node at (-2,4) [below]{${\scriptstyle 1}$};
\node at (2,4) [below]{${\scriptstyle k}$};
\node at (7,4) [below]{${\scriptstyle l}$};
\end{scope}
}\,
-\frac{[l]}{[k+1][k+l+2]}
\,\tikz[baseline=-.6ex, scale=.1, yscale=.7]{
\begin{scope}[yshift=3cm]
\draw[->-=.5] (-2,0) -- (-2,4);
\draw[->-=.5] (2,1) -- (2,4);
\draw[-<-=.5] (6,1) -- (6,4);
\draw[fill=white] (-1,0) rectangle (8,1);
\node at (-2,4) [above]{${\scriptstyle 1}$};
\node at (2,4) [above]{${\scriptstyle k}$};
\node at (6,4) [above]{${\scriptstyle l}$};
\end{scope}
\draw[->-=.5] (2,-3) -- (2,3);
\draw[-<-=.5] (6,-3) -- (6,3);
\draw (-2,3) to[out=south, in=west] (-1,1) to[out=east, in=south] (0,3);
\draw (-2,-3) to[out=north, in=west] (-1,-1) to[out=east, in=north](0,-3);
\node at (2,0) [right]{${\scriptstyle k}$};
\node at (6,0) [right]{${\scriptstyle l-1}$};
\begin{scope}[yshift=-3cm, yscale=-1]
\draw[-<-=.5] (-2,0) -- (-2,4);
\draw[-<-=.5] (2,1) -- (2,4);
\draw[->-=.5] (6,1) -- (6,4);
\draw[fill=white] (-1,0) rectangle (8,1);
\node at (-2,4) [below]{${\scriptstyle 1}$};
\node at (2,4) [below]{${\scriptstyle k}$};
\node at (6,4) [below]{${\scriptstyle l}$};
\end{scope}
}\, .
\end{align*}
\end{PROP}

\begin{proof}
We decompose $P_{{+}^{k+1}{-}^{l}}^{{+}^{k+1}{-}^{l}}$ into $(\mathbf{1}_{+}\otimes P_{{+}^{k}{-}^{l}}^{{+}^{k}{-}^{l}})P_{{+}^{k+1}{-}^{l}}^{{+}^{k+1}{-}^{l}}(\mathbf{1}_{+}\otimes P_{{+}^{k}{-}^{l}}^{{+}^{k}{-}^{l}})$ and expand the middle $A_2$ clasp $P_{{+}^{k+1}{-}^{l}}^{{+}^{k+1}{-}^{l}}$ by using Definition~\ref{OYdef}.
We calculate the following $A_2$ web:
\begin{align*}
X(k,l;i)
&=\,\tikz[baseline=-.6ex, scale=.1, yscale=.8]{
\draw[->-=.5] (0,-4) -- (0,4);
\draw[-<-=.5] (10,-9) -- (10,9);
\begin{scope}[yshift=3cm]
\draw[-<-=.5] (9,6) -- (9,9);
\draw[->-=.5] (2,6) -- (2,9);
\draw[->-=.5] (-2,0) -- (-2,9);
\draw[->-=.5] (2,1) -- (2,6);
\draw[-<-=.5] (8,1) -- (8,6);
\draw[-<-=.5, rounded corners] (4,0) -- (4,-1) -- (8,-1) -- (8,1);
\draw[fill=white] (0,5) rectangle (12,6);
\draw[fill=white] (-3,0) rectangle (5,1);
\end{scope}
\begin{scope}[yscale=-1, yshift=3cm]
\draw[->-=.5] (9,6) -- (9,9);
\draw[-<-=.5] (2,6) -- (2,9);
\draw[-<-=.5] (-2,0) -- (-2,9);
\draw[-<-=.5] (2,1) -- (2,6);
\draw[->-=.5] (8,1) -- (8,6);
\draw[->-=.5, rounded corners] (4,0) -- (4,-1) -- (8,-1) -- (8,1);
\draw[fill=white] (0,5) rectangle (12,6);
\draw[fill=white] (-3,0) rectangle (5,1);
\end{scope}
\node at (-2,12) [above]{${\scriptstyle 1}$};
\node at (-2,-12) [below]{${\scriptstyle 1}$};
\node at (2,12) [above]{${\scriptstyle k}$};
\node at (2,-12) [below]{${\scriptstyle k}$};
\node at (9,12) [above]{${\scriptstyle l}$};
\node at (9,-12) [below]{${\scriptstyle l}$};
\node at (2,6) [left]{${\scriptstyle k}$};
\node at (2,-6) [left]{${\scriptstyle k}$};
\node at (8,6) [left]{${\scriptstyle i}$};
\node at (8,-6) [left]{${\scriptstyle i}$};
\node at (10,0) [right]{${\scriptstyle l-i}$};
\node at (0,0) [left]{${\scriptstyle k+1-i}$};
}\,
=-\frac{[k]}{[k+1]}
\,\tikz[baseline=-.6ex, scale=.1, yscale=.8]{
\draw[->-=.5] (2,-4) -- (2,4);
\draw[-<-=.5] (11,-9) -- (11,9);
\begin{scope}[yshift=3cm]
\draw[-<-=.5] (1,5) -- (1,4) -- (-2,4);
\draw[->-=.5] (1,1) -- (1,2) -- (-2,2);
\draw[-<-=.5] (10,6) -- (10,9);
\draw[->-=.5] (2,6) -- (2,9);
\draw[->-=.2, ->-=.8] (-2,-6) -- (-2,9);
\draw[->-=.5] (3,1) -- (3,6);
\draw (9,1) -- (9,6);
\draw[-<-=.5, rounded corners] (4,0) -- (4,-1) -- (9,-1) -- (9,1);
\draw[fill=white] (0,5) rectangle (12,6);
\draw[fill=white] (-1,0) rectangle (5,1);
\end{scope}
\begin{scope}[yscale=-1, yshift=3cm]
\draw[->-=.5] (10,6) -- (10,9);
\draw[-<-=.5] (2,6) -- (2,9);
\draw[-<-=.5] (-2,0) -- (-2,9);
\draw[-<-=.5] (2,1) -- (2,6);
\draw[->-=.5] (9,1) -- (9,6);
\draw[->-=.5, rounded corners] (4,0) -- (4,-1) -- (9,-1) -- (9,1);
\draw[fill=white] (0,5) rectangle (12,6);
\draw[fill=white] (-3,0) rectangle (5,1);
\end{scope}
\node at (-2,12) [above]{${\scriptstyle 1}$};
\node at (-2,-12) [below]{${\scriptstyle 1}$};
\node at (2,12) [above]{${\scriptstyle k}$};
\node at (2,-12) [below]{${\scriptstyle k}$};
\node at (9,12) [above]{${\scriptstyle l}$};
\node at (9,-12) [below]{${\scriptstyle l}$};
\node at (2.5,6) [right]{${\scriptstyle k-1}$};
\node at (2,-6) [left]{${\scriptstyle k}$};
\node at (8,6) [below right]{${\scriptstyle i}$};
\node at (8,-6) [left]{${\scriptstyle i}$};
\node at (11,0) [right]{${\scriptstyle l-i}$};
\node at (2,0) [right]{${\scriptstyle k-i}$};
}\,
\end{align*}
By Lemma~{\ref{lem2}}, $X(k,l;i)=0$ for $i>1$ and
\begin{align*}
X(k,l;1)&=-\frac{[k]}{[k+1]}\frac{(-1)^{k-1}}{[k]}
\,\tikz[baseline=-.6ex, scale=.1, yscale=.8]{
\draw[->-=.5] (3,-4) -- (3,9);
\draw[-<-=.5] (11,-9) -- (11,9);
\begin{scope}[yshift=3cm]
\draw[-<-=.5] (1,5) -- (1,4) -- (-2,4);
\draw[->-=.8] (2,5) -- (2,2) -- (-2,2);
\draw[-<-=.5] (10,6) -- (10,9);
\draw[->-=.5] (2,6) -- (2,9);
\draw[->-=.2, ->-=.8] (-2,-6) -- (-2,9);
\draw[fill=white] (0,5) rectangle (12,6);
\end{scope}
\begin{scope}[yscale=-1, yshift=3cm]
\draw[->-=.5] (10,6) -- (10,9);
\draw[-<-=.5] (2,6) -- (2,9);
\draw[-<-=.5] (-2,0) -- (-2,9);
\draw[-<-=.5] (2,1) -- (2,6);
\draw[->-=.5] (9,1) -- (9,6);
\draw[->-=.5, rounded corners] (4,0) -- (4,-1) -- (9,-1) -- (9,1);
\draw[fill=white] (0,5) rectangle (12,6);
\draw[fill=white] (-3,0) rectangle (5,1);
\end{scope}
\node at (-2,12) [above]{${\scriptstyle 1}$};
\node at (-2,-12) [below]{${\scriptstyle 1}$};
\node at (2,12) [above]{${\scriptstyle k}$};
\node at (2,-12) [below]{${\scriptstyle k}$};
\node at (9,12) [above]{${\scriptstyle l}$};
\node at (9,-12) [below]{${\scriptstyle l}$};
\node at (2,-6) [left]{${\scriptstyle k}$};
\node at (8,-6) [left]{${\scriptstyle 1}$};
\node at (11,0) [right]{${\scriptstyle l-1}$};
\node at (3,3) [right]{${\scriptstyle k-1}$};
}\, \\
&=
\frac{1}{[k+1]^2}
\,\tikz[baseline=-.6ex, scale=.1, yscale=.8]{
\draw[->-=.5] (3,-9) -- (3,9);
\draw[-<-=.5] (11,-9) -- (11,9);
\draw[->-=.5] (-2,-5) -- (-2,5);
\begin{scope}[yshift=3cm]
\draw[-<-=.5] (1,5) -- (1,4) -- (-2,4);
\draw[->-=.8] (2,5) -- (2,2) -- (-2,2);
\draw[-<-=.5] (10,6) -- (10,9);
\draw[->-=.5] (2,6) -- (2,9);
\draw[->-=.8] (-2,2) -- (-2,9);
\draw[fill=white] (0,5) rectangle (12,6);
\end{scope}
\begin{scope}[yscale=-1, yshift=3cm]
\draw[->-=.5] (1,5) -- (1,4) -- (-2,4);
\draw[-<-=.8] (2,5) -- (2,2) -- (-2,2);
\draw[->-=.5] (10,6) -- (10,9);
\draw[-<-=.5] (2,6) -- (2,9);
\draw[-<-=.8] (-2,2) -- (-2,9);
\draw[fill=white] (0,5) rectangle (12,6);
\end{scope}
\node at (-2,12) [above]{${\scriptstyle 1}$};
\node at (-2,-12) [below]{${\scriptstyle 1}$};
\node at (2,12) [above]{${\scriptstyle k}$};
\node at (2,-12) [below]{${\scriptstyle k}$};
\node at (9,12) [above]{${\scriptstyle l}$};
\node at (9,-12) [below]{${\scriptstyle l}$};
\node at (11,0) [right]{${\scriptstyle l-1}$};
\node at (3,0) [right]{${\scriptstyle k-1}$};
}\,
=
\frac{1}{[k+1]^2}
\,\tikz[baseline=-.6ex, scale=.1, yscale=.8]{
\draw[->-=.5] (3,-9) -- (3,9);
\draw[-<-=.5] (9,-9) -- (9,9);
\begin{scope}[yshift=3cm]
\draw[->-=.8, rounded corners] (1,5) -- (1,2) -- (-2,2) -- (-2,9);
\draw[-<-=.5] (8,6) -- (8,9);
\draw[->-=.5] (2,6) -- (2,9);
\draw[fill=white] (0,5) rectangle (12,6);
\end{scope}
\begin{scope}[yscale=-1, yshift=3cm]
\draw[-<-=.8, rounded corners] (1,5) -- (1,2) -- (-2,2) -- (-2,9);
\draw[->-=.5] (8,6) -- (8,9);
\draw[-<-=.5] (2,6) -- (2,9);
\draw[fill=white] (0,5) rectangle (12,6);
\end{scope}
\node at (-2,12) [above]{${\scriptstyle 1}$};
\node at (-2,-12) [below]{${\scriptstyle 1}$};
\node at (2,12) [above]{${\scriptstyle k}$};
\node at (2,-12) [below]{${\scriptstyle k}$};
\node at (8,12) [above]{${\scriptstyle l}$};
\node at (8,-12) [below]{${\scriptstyle l}$};
\node at (9,0) [right]{${\scriptstyle l-1}$};
\node at (3,0) [right]{${\scriptstyle k}$};
}\, .
\end{align*}
We can easily calculate $X(k,l;0)$ from Definition~{\ref{singledef}}:
\begin{align*}
X(k,l;0)
&=(\mathbf{1}_{+}\otimes P_{{+}^k{-}^l}^{{+}^k{-}^l})(P_{{+}^{k+1}}^{{+}^{k+1}}\otimes\mathbf{1}_{+})(\mathbf{1}_{+}\otimes P_{{+}^k{-}^l}^{{+}^k{-}^l})\\
&=\mathbf{1}_{+}\otimes P_{{+}^{k}{-}^{l}}^{{+}^{k}{-}^{l}}
-\frac{[k]}{[k+1]}
(\mathbf{1}_{+}\otimes P_{{+}^{k}{-}^{l}}^{{+}^{k}{-}^{l}})
(t_{{-}}^{{+}{+}}t_{{+}{+}}^{{-}}\otimes\mathbf{1}_{{+}^{k-1}{-}^{l}})(\mathbf{1}_{+}\otimes P_{{+}^{k}{-}^{l}}^{{+}^{k}{-}^{l}})
\end{align*}
We complete the proof by substituting the above solutions of $X(k,l;i)$ into
\[
P_{{+}^{k+1}{-}^{l}}^{{+}^{k+1}{-}^{l}}
=\sum_{i=0}^{\min\{k+1,l\}}
(-1)^i
\frac{{k+1\brack i}{l\brack i}}{{k+l+2\brack i}}
X(k,l;i)
=X(k,l;0)-\frac{[k+1][l]}{[k+l+2]}X(k,l;1).
\]
\end{proof}

\begin{LEM}\label{lem7}
\begin{align*}
\,\tikz[baseline=-.6ex, scale=.1]{
\begin{scope}[yshift=5cm]
\draw[-<-=.5] (-2,0) -- (-2,4);
\draw[->-=.5] (2,1) -- (2,4);
\draw[-<-=.5] (7,1) -- (7,4);
\draw[fill=white] (0,0) rectangle (9,1);
\node at (2,4) [above]{${\scriptstyle k-1}$};
\node at (7,4) [above]{${\scriptstyle l}$};
\end{scope}
\draw (-2,-3) -- (-2,5);
\draw[->-=.5] (0,0) -- (0,3) -- (-2,3);
\draw[->-=.5] (0,0) -- (0,-3) -- (-2,-3);
\draw[->-=.2, ->-=.8] (2,-5) -- (2,5);
\draw[-<-=.2, -<-=.8] (7,-5) -- (7,5);
\draw[fill=white] (-1,0) rectangle (9,1);
\begin{scope}[yshift=-5cm, yscale=-1]
\draw[-<-=.5] (2,1) -- (2,4);
\draw[->-=.5] (7,1) -- (7,4);
\draw[fill=white] (0,0) rectangle (9,1);
\node at (2,4) [below]{${\scriptstyle k}$};
\node at (7,4) [below]{${\scriptstyle l-1}$};
\end{scope}
}\,
=
\frac{[k+1]}{[k]}\left(1-\frac{[l+1]}{[k+1]{[k+l+1]}}\right)
\,\tikz[baseline=-.6ex, scale=.1]{
\begin{scope}[yshift=5cm]
\draw[-<-=.5] (-2,0) -- (-2,4);
\draw[->-=.5] (2,1) -- (2,4);
\draw[-<-=.5] (7,1) -- (7,4);
\draw[fill=white] (-1,0) rectangle (9,1);
\node at (2,4) [above]{${\scriptstyle k-1}$};
\node at (7,4) [above]{${\scriptstyle l}$};
\end{scope}
\draw[->-=.5] (-2,-5) -- (-2,5);
\draw[->-=.5] (0,5) -- (0,3) -- (-2,3);
\draw[->-=.5] (3,-5) -- (3,5);
\draw[-<-=.5] (7,-5) -- (7,5);
\begin{scope}[yshift=-5cm, yscale=-1]
\draw[-<-=.5] (2,1) -- (2,4);
\draw[->-=.5] (7,1) -- (7,4);
\draw[fill=white] (-3,0) rectangle (9,1);
\node at (2,4) [below]{${\scriptstyle k}$};
\node at (7,4) [below]{${\scriptstyle l-1}$};
\end{scope}
}\,
\end{align*}
\end{LEM}

\begin{proof}
Apply Proposition~\ref{prop5} to $P_{{+}^{k}{-}^{l}}^{{+}^{k}{-}^{l}}$ as follows:
\begin{align*}
\,\tikz[baseline=-.6ex, scale=.1]{
\begin{scope}[yshift=5cm]
\draw[-<-=.5] (-2,0) -- (-2,4);
\draw[->-=.5] (2,1) -- (2,4);
\draw[-<-=.5] (7,1) -- (7,4);
\draw[fill=white] (0,0) rectangle (9,1);
\node at (2,4) [above]{${\scriptstyle k-1}$};
\node at (7,4) [above]{${\scriptstyle l}$};
\end{scope}
\draw (-2,-3) -- (-2,5);
\draw[->-=.5] (0,0) -- (0,3) -- (-2,3);
\draw[->-=.5] (0,0) -- (0,-3) -- (-2,-3);
\draw[->-=.2, ->-=.8] (2,-5) -- (2,5);
\draw[-<-=.2, -<-=.8] (7,-5) -- (7,5);
\draw[fill=white] (-1,0) rectangle (9,1);
\begin{scope}[yshift=-5cm, yscale=-1]
\draw[-<-=.5] (2,1) -- (2,4);
\draw[->-=.5] (7,1) -- (7,4);
\draw[fill=white] (0,0) rectangle (9,1);
\node at (2,4) [below]{${\scriptstyle k}$};
\node at (7,4) [below]{${\scriptstyle l-1}$};
\end{scope}
}\, 
&=
\,\tikz[baseline=-.6ex, scale=.1]{
\begin{scope}[yshift=5cm]
\draw[-<-=.5] (-2,0) -- (-2,4);
\draw[->-=.5] (2,1) -- (2,4);
\draw[-<-=.5] (6,1) -- (6,4);
\draw[fill=white] (0,0) rectangle (9,1);
\node at (2,4) [above]{${\scriptstyle k-1}$};
\node at (6,4) [above]{${\scriptstyle l}$};
\end{scope}
\draw[->-=.2] (-2,-2) -- (-2,5);
\draw[->-=.5] (0,0) -- (0,3) -- (-2,3);
\draw (5,0) -- (5,-2) -- (-2,-2);
\draw[->-=.5] (2,0) -- (2,5);
\draw[->-=.2] (2,-5) -- (2,0);
\draw[-<-=.2, -<-=.8] (7,-5) -- (7,5);
\draw[fill=white] (-1,0) rectangle (9,1);
\draw[fill=white] (0,-3) rectangle (4,-1);
\draw (0,-3) -- (4,-1);
\begin{scope}[yshift=-5cm, yscale=-1]
\draw[-<-=.5] (2,1) -- (2,4);
\draw[->-=.5] (7,1) -- (7,4);
\draw[fill=white] (0,0) rectangle (9,1);
\node at (2,4) [below]{${\scriptstyle k}$};
\node at (7,4) [below]{${\scriptstyle l-1}$};
\end{scope}
}\,
=
\,\tikz[baseline=-.6ex, scale=.1]{
\begin{scope}[yshift=5cm]
\draw[-<-=.5] (-3,0) -- (-3,4);
\draw[->-=.5] (2,1) -- (2,4);
\draw[-<-=.5] (6,1) -- (6,4);
\draw[fill=white] (-2,0) rectangle (9,1);
\node at (2,4) [above]{${\scriptstyle k-1}$};
\node at (6,4) [above]{${\scriptstyle l}$};
\end{scope}
\draw[->-=.5] (-3,-1) -- (-3,5);
\draw[->-=.2, ->-=.8] (-2,-5) -- (-2,3);
\draw (-3,3) -- (-2,3);
\draw (-3,-1) -- (-2,-1);
\draw[-<-=.8] (-2,-2) -- (5,-2) -- (5,5);
\draw[->-=.5] (2,0) -- (2,5);
\draw[->-=.2] (2,-5) -- (2,0);
\draw[-<-=.5] (7,-5) -- (7,5);
\draw[fill=white] (0,-3) rectangle (4,-1);
\draw (0,-3) -- (4,-1);
\begin{scope}[yshift=-5cm, yscale=-1]
\draw[-<-=.5] (0,1) -- (0,4);
\draw[->-=.5] (7,1) -- (7,4);
\draw[fill=white] (-4,0) rectangle (9,1);
\node at (0,4) [below]{${\scriptstyle k}$};
\node at (7,4) [below]{${\scriptstyle l-1}$};
\end{scope}
}\,
-\frac{[k-1]}{[k]}
\,\tikz[baseline=-.6ex, scale=.1]{
\begin{scope}[yshift=5cm]
\draw[-<-=.5] (-3,0) -- (-3,4);
\draw[->-=.5] (1,1) -- (1,4);
\draw[-<-=.5] (6,1) -- (6,4);
\draw[fill=white] (-2,0) rectangle (9,1);
\node at (1,4) [above]{${\scriptstyle k-1}$};
\node at (6,4) [above]{${\scriptstyle l}$};
\end{scope}
\draw (-3,1) -- (-3,5);
\draw[->-=.2, -<-=.5] (-2,-5) -- (-2,4);
\draw[->-=.7] (-3,4) -- (0,4) -- (0,5);
\draw[-<-=.5] (-2,2) -- (0,2) -- (0,1);
\draw (-3,1) -- (-2,1);
\draw[-<-=.9] (-2,-2) -- (5,-2) -- (5,0);
\draw[->-=.5] (2,0) -- (2,5);
\draw[->-=.2] (2,-5) -- (2,0);
\draw[-<-=.2, -<-=.8] (7,-5) -- (7,5);
\draw[fill=white] (-1,0) rectangle (9,1);
\draw[fill=white] (0,-3) rectangle (4,-1);
\draw (0,-3) -- (4,-1);
\begin{scope}[yshift=-5cm, yscale=-1]
\draw[-<-=.5] (0,1) -- (0,4);
\draw[->-=.5] (7,1) -- (7,4);
\draw[fill=white] (-4,0) rectangle (9,1);
\node at (0,4) [below]{${\scriptstyle k}$};
\node at (7,4) [below]{${\scriptstyle l-1}$};
\end{scope}
}\,
-\frac{[l]}{[k][k+l+1]}
\,\tikz[baseline=-.6ex, scale=.1]{
\begin{scope}[yshift=5cm]
\draw[-<-=.5] (-3,0) -- (-3,4);
\draw[->-=.5] (2,1) -- (2,4);
\draw[-<-=.5] (6,1) -- (6,4);
\draw[fill=white] (-2,0) rectangle (9,1);
\node at (2,4) [above]{${\scriptstyle k-1}$};
\node at (6,4) [above]{${\scriptstyle l}$};
\end{scope}
\draw (-3,2) -- (-3,5);
\draw[->-=.2, -<-=.8] (-3,-5) -- (-3,2);
\draw[-<-=.5] (-3,4) -- (0,4) -- (0,5);
\draw[->-=.5] (-3,2) -- (0,2) -- (0,1);
\draw[-<-=.9] (-3,-2) -- (5,-2) -- (5,0);
\draw[->-=.5] (2,0) -- (2,5);
\draw[->-=.2] (2,-5) -- (2,0);
\draw[-<-=.2, -<-=.8] (7,-5) -- (7,5);
\draw[fill=white] (-1,0) rectangle (9,1);\
\draw[fill=white] (0,-3) rectangle (4,-1);
\draw (0,-3) -- (4,-1);
\begin{scope}[yshift=-5cm, yscale=-1]
\draw[-<-=.5] (0,1) -- (0,4);
\draw[->-=.5] (7,1) -- (7,4);
\draw[fill=white] (-4,0) rectangle (9,1);
\node at (0,4) [below]{${\scriptstyle k}$};
\node at (7,4) [below]{${\scriptstyle l-1}$};
\end{scope}
}\,\\
&=\left([2]-\frac{[k-1]}{[k]}\right)
\,\tikz[baseline=-.6ex, scale=.1]{
\begin{scope}[yshift=5cm]
\draw[-<-=.5] (-3,0) -- (-3,4);
\draw[->-=.5] (2,1) -- (2,4);
\draw[-<-=.5] (6,1) -- (6,4);
\draw[fill=white] (-2,0) rectangle (9,1);
\node at (2,4) [above]{${\scriptstyle k-1}$};
\node at (6,4) [above]{${\scriptstyle l}$};
\end{scope}
\draw[->-=.5] (-3,-5) -- (-3,5);
\draw[-<-=.8] (-3,1) -- (5,1) -- (5,5);
\draw[->-=.5] (2,2) -- (2,5);
\draw[->-=.5] (2,-5) -- (2,2);
\draw[-<-=.5] (7,-5) -- (7,5);
\draw[fill=white] (0,0) rectangle (4,2);
\draw (0,0) -- (4,2);
\begin{scope}[yshift=-5cm, yscale=-1]
\draw[-<-=.5] (0,1) -- (0,4);
\draw[->-=.5] (7,1) -- (7,4);
\draw[fill=white] (-4,0) rectangle (9,1);
\node at (0,4) [below]{${\scriptstyle k}$};
\node at (7,4) [below]{${\scriptstyle l-1}$};
\end{scope}
}\,
-\frac{[k-1]}{[k]}
\,\tikz[baseline=-.6ex, scale=.1]{
\begin{scope}[yshift=5cm]
\draw[-<-=.5] (-2,-1) -- (-2,4);
\draw[->-=.5] (1,1) -- (1,4);
\draw[-<-=.5] (6,1) -- (6,4);
\draw[fill=white] (-1,0) rectangle (9,1);
\node at (1,4) [above]{${\scriptstyle k-1}$};
\node at (6,4) [above]{${\scriptstyle l}$};
\end{scope}
\draw[->-=.2, -<-=.8] (-2,-5) -- (-2,2);
\draw (-2,4) -- (0,4) -- (0,5);
\draw (-2,2) -- (0,2) -- (0,1);
\draw[-<-=.9] (-2,-2) -- (5,-2) -- (5,0);
\draw[->-=.5] (2,0) -- (2,5);
\draw[->-=.2] (2,-5) -- (2,0);
\draw[-<-=.2, -<-=.8] (7,-5) -- (7,5);
\draw[fill=white] (-1,0) rectangle (9,1);
\draw[fill=white] (0,-3) rectangle (4,-1);
\draw (0,-3) -- (4,-1);
\begin{scope}[yshift=-5cm, yscale=-1]
\draw[-<-=.5] (0,1) -- (0,4);
\draw[->-=.5] (7,1) -- (7,4);
\draw[fill=white] (-4,0) rectangle (9,1);
\node at (0,4) [below]{${\scriptstyle k}$};
\node at (7,4) [below]{${\scriptstyle l-1}$};
\end{scope}
}\,
-\frac{[l]}{[k][k+l+1]}
\,\tikz[baseline=-.6ex, scale=.1]{
\begin{scope}[yshift=5cm]
\draw[-<-=.5] (-3,0) -- (-3,4);
\draw[->-=.5] (2,1) -- (2,4);
\draw[-<-=.5] (6,1) -- (6,4);
\draw[fill=white] (-2,0) rectangle (9,1);
\node at (2,4) [above]{${\scriptstyle k-1}$};
\node at (6,4) [above]{${\scriptstyle l}$};
\end{scope}
\draw (-3,2) -- (-3,5);
\draw[->-=.2, -<-=.8] (-3,-5) -- (-3,2);
\draw[-<-=.5] (-3,4) -- (0,4) -- (0,5);
\draw[->-=.5] (-3,2) -- (0,2) -- (0,1);
\draw[-<-=.9] (-3,-2) -- (5,-2) -- (5,0);
\draw[->-=.5] (2,0) -- (2,5);
\draw[->-=.2] (2,-5) -- (2,0);
\draw[-<-=.2, -<-=.8] (7,-5) -- (7,5);
\draw[fill=white] (-1,0) rectangle (9,1);\
\draw[fill=white] (0,-3) rectangle (4,-1);
\draw (0,-3) -- (4,-1);
\begin{scope}[yshift=-5cm, yscale=-1]
\draw[-<-=.5] (0,1) -- (0,4);
\draw[->-=.5] (7,1) -- (7,4);
\draw[fill=white] (-4,0) rectangle (9,1);
\node at (0,4) [below]{${\scriptstyle k}$};
\node at (7,4) [below]{${\scriptstyle l-1}$};
\end{scope}
}\,\\
&=\frac{[k+1]}{[k]}
\,\tikz[baseline=-.6ex, scale=.1]{
\begin{scope}[yshift=5cm]
\draw[-<-=.5] (-3,0) -- (-3,4);
\draw[->-=.5] (2,1) -- (2,4);
\draw[-<-=.5] (6,1) -- (6,4);
\draw[fill=white] (-2,0) rectangle (9,1);
\node at (2,4) [above]{${\scriptstyle k-1}$};
\node at (6,4) [above]{${\scriptstyle l}$};
\end{scope}
\draw[->-=.5] (-3,-5) -- (-3,5);
\draw[-<-=.5] (-3,4) -- (0,4) -- (0,5);
\draw[->-=.5] (2,-5) -- (2,5);
\draw[-<-=.5] (7,-5) -- (7,5);
\begin{scope}[yshift=-5cm, yscale=-1]
\draw[-<-=.5] (0,1) -- (0,4);
\draw[->-=.5] (7,1) -- (7,4);
\draw[fill=white] (-4,0) rectangle (9,1);
\node at (0,4) [below]{${\scriptstyle k}$};
\node at (7,4) [below]{${\scriptstyle l-1}$};
\end{scope}
}\,
-\frac{[l]}{[k][k+l+1]}
\,\tikz[baseline=-.6ex, scale=.1]{
\begin{scope}[yshift=5cm]
\draw[-<-=.5] (-3,0) -- (-3,4);
\draw[->-=.5] (2,1) -- (2,4);
\draw[-<-=.5] (6,1) -- (6,4);
\draw[fill=white] (-2,0) rectangle (9,1);
\node at (2,4) [above]{${\scriptstyle k-1}$};
\node at (6,4) [above]{${\scriptstyle l}$};
\end{scope}
\draw (-3,2) -- (-3,5);
\draw[->-=.2, -<-=.8] (-3,-5) -- (-3,2);
\draw[-<-=.5] (-3,4) -- (0,4) -- (0,5);
\draw[->-=.5] (-3,2) -- (0,2) -- (0,1);
\draw[-<-=.5] (-3,-1) -- (0,-1) -- (0,0);
\draw[->-=.5] (2,0) -- (2,5);
\draw[->-=.5] (2,-5) -- (2,0);
\draw[-<-=.2, -<-=.8] (7,-5) -- (7,5);
\draw[fill=white] (-1,0) rectangle (9,1);
\begin{scope}[yshift=-5cm, yscale=-1]
\draw[-<-=.5] (0,1) -- (0,4);
\draw[->-=.5] (7,1) -- (7,4);
\draw[fill=white] (-4,0) rectangle (9,1);
\node at (0,4) [below]{${\scriptstyle k}$};
\node at (7,4) [below]{${\scriptstyle l-1}$};
\end{scope}
}\, .
\end{align*}
Once again, we apply Proposition~\ref{prop5} to $P_{{-}^{l}{+}^{k-1}}^{{-}^{l}{+}^{k-1}}$:
\begin{align}\label{eqprop5}
\,\tikz[baseline=-.6ex, scale=.1]{
\begin{scope}[yshift=5cm]
\draw[-<-=.5] (-3,0) -- (-3,4);
\draw[->-=.5] (2,1) -- (2,4);
\draw[-<-=.5] (8,1) -- (8,4);
\draw[fill=white] (-2,0) rectangle (11,1);
\node at (2,4) [above]{${\scriptstyle k-1}$};
\node at (8,4) [above]{${\scriptstyle l}$};
\end{scope}
\draw (-3,2) -- (-3,5);
\draw[->-=.2, -<-=.8] (-3,-5) -- (-3,2);
\draw[-<-=.5] (-3,4) -- (0,4) -- (0,5);
\draw[->-=.5] (-3,2) -- (0,2) -- (0,1);
\draw[-<-=.5] (-3,-1) -- (0,-1) -- (0,0);
\draw[->-=.5] (2,0) -- (2,5);
\draw[->-=.5] (2,-5) -- (2,0);
\draw[-<-=.2, -<-=.8] (9,-5) -- (9,5);
\draw[fill=white] (-1,0) rectangle (11,1);
\node at (2,-2) [right]{${\scriptstyle k-1}$};
\node at (9,-2) [right]{${\scriptstyle l-1}$};
\begin{scope}[yshift=-5cm, yscale=-1]
\draw[-<-=.5] (0,1) -- (0,4);
\draw[->-=.5] (9,1) -- (9,4);
\draw[fill=white] (-4,0) rectangle (11,1);
\node at (0,4) [below]{${\scriptstyle k}$};
\node at (9,4) [below]{${\scriptstyle l-1}$};
\end{scope}
}\,
&=
\,\tikz[baseline=-.6ex, scale=.1]{
\begin{scope}[yshift=5cm]
\draw[-<-=.5] (-3,0) -- (-3,4);
\draw[->-=.5] (1,1) -- (1,4);
\draw[-<-=.5] (8,1) -- (8,4);
\draw[fill=white] (-2,0) rectangle (11,1);
\node at (1,4) [above]{${\scriptstyle k-1}$};
\node at (8,4) [above]{${\scriptstyle l}$};
\end{scope}
\draw (-3,2) -- (-3,5);
\draw[->-=.2, -<-=.8] (-3,-5) -- (-3,2);
\draw[-<-=.5] (-3,4) -- (0,4) -- (0,5);
\draw[->-=.5] (-3,2) -- (0,2) -- (0,1);
\draw[-<-=.5] (-3,-1) -- (0,-1) -- (0,0);
\draw[-<-=.5] (2,0) -- (2,5);
\draw[-<-=.5] (2,-5) -- (2,0);
\draw[->-=.2, ->-=.8] (9,-5) -- (9,5);
\draw[fill=white] (-1,0) rectangle (11,1);
\node at (2,-2) [right]{${\scriptstyle l-1}$};
\node at (9,-2) [right]{${\scriptstyle k-1}$};
\begin{scope}[yshift=-5cm, yscale=-1]
\draw[-<-=.5] (0,1) -- (0,4);
\draw[->-=.5] (9,1) -- (9,4);
\draw[fill=white] (-4,0) rectangle (11,1);
\node at (0,4) [below]{${\scriptstyle k}$};
\node at (9,4) [below]{${\scriptstyle l-1}$};
\end{scope}
}\, 
=
\,\tikz[baseline=-.6ex, scale=.1]{
\begin{scope}[yshift=5cm]
\draw[-<-=.5] (-3,0) -- (-3,4);
\draw[->-=.5] (1,1) -- (1,4);
\draw[-<-=.5] (8,1) -- (8,4);
\draw[fill=white] (-2,0) rectangle (11,1);
\node at (1,4) [above]{${\scriptstyle k-1}$};
\node at (8,4) [above]{${\scriptstyle l}$};
\end{scope}
\draw (-3,2) -- (-3,5);
\draw[->-=.2, -<-=.8] (-3,-5) -- (-3,2);
\draw[-<-=.5] (-3,4) -- (0,4) -- (0,5);
\draw[->-=.5] (-3,2) -- (0,2) -- (0,-1) -- (-3,-1);
\draw[-<-=.5] (2,0) -- (2,5);
\draw (2,-5) -- (2,0);
\draw[->-=.8] (9,-5) -- (9,5);
\node at (2,-2) [right]{${\scriptstyle l-1}$};
\node at (9,-2) [right]{${\scriptstyle k-1}$};
\begin{scope}[yshift=-5cm, yscale=-1]
\draw[-<-=.5] (0,1) -- (0,4);
\draw[->-=.5] (9,1) -- (9,4);
\draw[fill=white] (-4,0) rectangle (11,1);
\node at (0,4) [below]{${\scriptstyle k}$};
\node at (9,4) [below]{${\scriptstyle l-1}$};
\end{scope}
}\,
-\frac{[l-1]}{[l]}
\,\tikz[baseline=-.6ex, scale=.1]{
\begin{scope}[yshift=5cm]
\draw[-<-=.5] (-3,0) -- (-3,4);
\draw[->-=.5] (1,1) -- (1,4);
\draw[-<-=.5] (8,1) -- (8,4);
\draw[fill=white] (-2,0) rectangle (11,1);
\node at (1,4) [above]{${\scriptstyle k-1}$};
\node at (8,4) [above]{${\scriptstyle l}$};
\end{scope}
\draw (-3,2) -- (-3,5);
\draw[->-=.2, -<-=.8] (-3,-5) -- (-3,2);
\draw[-<-=.5] (-3,4) -- (-1,4) -- (-1,5);
\draw[-<-=.5] (0,2) -- (0,5);
\draw[->-=.5] (0,-1) -- (0,-5);
\draw[-<-=.5] (-3,2) -- (0,2) -- (0,-1) -- (-3,-1);
\draw[-<-=.5] (2,0) -- (2,5);
\draw (2,-5) -- (2,0);
\draw[->-=.8] (9,-5) -- (9,5);
\node at (2,-2) [right]{${\scriptstyle l-2}$};
\node at (9,-2) [right]{${\scriptstyle k-1}$};
\begin{scope}[yshift=-5cm, yscale=-1]
\draw[-<-=.5] (0,1) -- (0,4);
\draw[->-=.5] (9,1) -- (9,4);
\draw[fill=white] (-4,0) rectangle (11,1);
\node at (0,4) [below]{${\scriptstyle k}$};
\node at (9,4) [below]{${\scriptstyle l-1}$};
\end{scope}
}\,\\
&=\frac{[l+1]}{[l]}
\,\tikz[baseline=-.6ex, scale=.1]{
\begin{scope}[yshift=5cm]
\draw[-<-=.5] (-3,0) -- (-3,4);
\draw[->-=.5] (1,1) -- (1,4);
\draw[-<-=.5] (8,1) -- (8,4);
\draw[fill=white] (-2,0) rectangle (11,1);
\node at (1,4) [above]{${\scriptstyle k-1}$};
\node at (8,4) [above]{${\scriptstyle l}$};
\end{scope}
\draw (-3,2) -- (-3,5);
\draw[->-=.5] (-3,-5) -- (-3,2);
\draw[-<-=.5] (-3,4) -- (0,4) -- (0,5);
\draw[-<-=.5] (2,0) -- (2,5);
\draw (2,-5) -- (2,0);
\draw[->-=.8] (9,-5) -- (9,5);
\node at (2,-2) [right]{${\scriptstyle l-1}$};
\node at (9,-2) [right]{${\scriptstyle k-1}$};
\begin{scope}[yshift=-5cm, yscale=-1]
\draw[-<-=.5] (0,1) -- (0,4);
\draw[->-=.5] (9,1) -- (9,4);
\draw[fill=white] (-4,0) rectangle (11,1);
\node at (0,4) [below]{${\scriptstyle k}$};
\node at (9,4) [below]{${\scriptstyle l-1}$};
\end{scope}
}\,
=\frac{[l+1]}{[l]}
\,\tikz[baseline=-.6ex, scale=.1]{
\begin{scope}[yshift=5cm]
\draw[-<-=.5] (-3,0) -- (-3,4);
\draw[->-=.5] (2,1) -- (2,4);
\draw[-<-=.5] (6,1) -- (6,4);
\draw[fill=white] (-2,0) rectangle (9,1);
\node at (2,4) [above]{${\scriptstyle k-1}$};
\node at (6,4) [above]{${\scriptstyle l}$};
\end{scope}
\draw[->-=.5] (-3,-5) -- (-3,5);
\draw[-<-=.5] (-3,4) -- (0,4) -- (0,5);
\draw[->-=.5] (2,-5) -- (2,5);
\draw[-<-=.5] (7,-5) -- (7,5);
\begin{scope}[yshift=-5cm, yscale=-1]
\draw[-<-=.5] (0,1) -- (0,4);
\draw[->-=.5] (7,1) -- (7,4);
\draw[fill=white] (-4,0) rectangle (9,1);
\node at (0,4) [below]{${\scriptstyle k}$};
\node at (7,4) [below]{${\scriptstyle l-1}$};
\end{scope}
}\, .\notag
\end{align}
\end{proof}

\begin{LEM}[Otsuki-Yamada~{\cite[Lemma~5.3]{OhtsukiYamada97}}]\label{lemOY}
\begin{align*}
\,\tikz[baseline=-.6ex, scale=.1]{
\begin{scope}[yshift=5cm]
\draw[->-=.5] (2,1) -- (2,4);
\draw[-<-=.5] (7,1) -- (7,4);
\draw[fill=white] (0,0) rectangle (9,1);
\node at (2,4) [above]{${\scriptstyle k}$};
\node at (7,4) [above]{${\scriptstyle l}$};
\end{scope}
\draw[-<-=.5] (-2,-3) -- (-2,3);
\draw (0,0) -- (0,3) -- (-2,3);
\draw (0,0) -- (0,-3) -- (-2,-3);
\draw[->-=.2, ->-=.8] (2,-5) -- (2,5);
\draw[-<-=.2, -<-=.8] (7,-5) -- (7,5);
\draw[fill=white] (-1,0) rectangle (9,1);
\node at (-2,0) [left]{${\scriptstyle 1}$};
\begin{scope}[yshift=-5cm, yscale=-1]
\draw[-<-=.5] (2,1) -- (2,4);
\draw[->-=.5] (7,1) -- (7,4);
\draw[fill=white] (0,0) rectangle (9,1);
\node at (2,4) [below]{${\scriptstyle k}$};
\node at (7,4) [below]{${\scriptstyle l}$};
\end{scope}
}\,
=
\frac{[k+2][k+l+3]}{[k+1][k+l+2]}P_{{+}^{k}{-}^{l}}^{{+}^{k}{-}^{l}}
\end{align*}
\end{LEM}
\begin{proof}
The $A_2$ web in the left-hand side should be expressed as a scalar multiplication of $P_{{+}^{k}{-}^{l}}^{{+}^{k}{-}^{l}}$.
This constant can be calculated by taking the closure of the $A_2$ webs in the both sides. We remark that the value of the closure of $P_{{+}^{k}{-}^{l}}^{{+}^{k}{-}^{l}}$ is $\frac{[k+1][l+1][k+l+2]}{[2]}$.
\end{proof}

\begin{PROP}\label{CK3}
For any positive integers $k,l\geq 1$,
\[
\langle P_{(k+1,l)}\rangle=X\langle P_{(k,l)}\rangle-\langle P_{(k-1,l+1)}\rangle-\langle P_{(k,l-1)}\rangle.
\]
\end{PROP}

\begin{proof}
By Proposition~\ref{prop5}, 
\begin{align}\label{eqCK3}
\mathbf{1}_{+}\otimes P_{{+}^{k}{-}^{l}}^{{+}^{k}{-}^{l}}
-P_{{+}^{k+1}{-}^{l}}^{{+}^{k+1}{-}^{l}}
=
\frac{[k]}{[k+1]}
\,\tikz[baseline=-.6ex, scale=.1, yscale=.7]{
\begin{scope}[yshift=3cm]
\draw[->-=.5] (-2,0) -- (-2,4);
\draw[->-=.5] (2,1) -- (2,4);
\draw[-<-=.5] (7,1) -- (7,4);
\draw[fill=white] (-1,0) rectangle (9,1);
\node at (-2,4) [above]{${\scriptstyle 1}$};
\node at (2,4) [above]{${\scriptstyle k}$};
\node at (7,4) [above]{${\scriptstyle l}$};
\end{scope}
\draw (-2,-3) -- (-2,3);
\draw[->-=.5] (1,-3) -- (1,3);
\draw[-<-=.5] (8,-3) -- (8,3);
\draw[->-=.5] (-2,1) -- (0,1) -- (0,3);
\draw[-<-=.5] (-2,-1) -- (0,-1) -- (0,-3);
\node at (.5,0) [right]{${\scriptstyle k-1}$};
\node at (8,0) [right]{${\scriptstyle l}$};
\begin{scope}[yshift=-3cm, yscale=-1]
\draw[-<-=.5] (-2,0) -- (-2,4);
\draw[-<-=.5] (2,1) -- (2,4);
\draw[->-=.5] (7,1) -- (7,4);
\draw[fill=white] (-1,0) rectangle (9,1);
\node at (-2,4) [below]{${\scriptstyle 1}$};
\node at (2,4) [below]{${\scriptstyle k}$};
\node at (7,4) [below]{${\scriptstyle l}$};
\end{scope}
}\,
+\frac{[l]}{[k+1][k+l+2]}
\,\tikz[baseline=-.6ex, scale=.1, yscale=.7]{
\begin{scope}[yshift=3cm]
\draw[->-=.5] (-2,0) -- (-2,4);
\draw[->-=.5] (2,1) -- (2,4);
\draw[-<-=.5] (6,1) -- (6,4);
\draw[fill=white] (-1,0) rectangle (8,1);
\node at (-2,4) [above]{${\scriptstyle 1}$};
\node at (2,4) [above]{${\scriptstyle k}$};
\node at (6,4) [above]{${\scriptstyle l}$};
\end{scope}
\draw[->-=.5] (2,-3) -- (2,3);
\draw[-<-=.5] (6,-3) -- (6,3);
\draw (-2,3) to[out=south, in=west] (-1,1) to[out=east, in=south] (0,3);
\draw (-2,-3) to[out=north, in=west] (-1,-1) to[out=east, in=north](0,-3);
\node at (2,0) [right]{${\scriptstyle k}$};
\node at (6,0) [right]{${\scriptstyle l-1}$};
\begin{scope}[yshift=-3cm, yscale=-1]
\draw[-<-=.5] (-2,0) -- (-2,4);
\draw[-<-=.5] (2,1) -- (2,4);
\draw[->-=.5] (6,1) -- (6,4);
\draw[fill=white] (-1,0) rectangle (8,1);
\node at (-2,4) [below]{${\scriptstyle 1}$};
\node at (2,4) [below]{${\scriptstyle k}$};
\node at (6,4) [below]{${\scriptstyle l}$};
\end{scope}
}\,.
\end{align}
It is easy to see that the left-hand side is an idempotent and 
$X\langle P_{(k,l)}\rangle-\langle P_{(k+1,l)}\rangle$ in $K_0(\Kar)$.
We denote the right-hand side of the above $A_2$ web by $g$.
We consider morphisms $p_1\colon ({+}^{k+1}{-}^{l},g)\to P_{(k-1,l+1)}$ and $p_2\colon ({+}^{k+1}{-}^{l},g)\to P_{(k,l-1)}$ defined by 
\begin{align*}
p_1
&=\frac{[k]}{[k+1]}P_{{+}^{k-1}{-}^{l+1}}^{{+}^{k-1}{-}^{l+1}}
(t_{{+}{+}}^{{-}}\otimes P_{{+}^{k-1}{-}^{l}}^{{+}^{k-1}{-}^{l}})
(\mathbf{1}_{+}\otimes P_{{+}^{k}{-}^{l}}^{{+}^{k}{-}^{l}}),\\
p_2
&=\frac{[l][k+l+1]}{[l+1][k+l+2]}(d_{{+}{-}}\otimes P_{{+}^{k}{-}^{l-1}}^{{+}^{k}{-}^{l-1}})
(\mathbf{1}_{+}\otimes P_{{+}^{k}{-}^{l}}^{{+}^{k}{-}^{l}}).
\end{align*}
Let us confirm that $p_1$ and $p_2$ are morphisms in $\Kar$.
By a similar way to (\ref{eqprop5}) and using Lemma~\ref{lem7},
\begin{align*}
&P_{{+}^{k-1}{-}^{l+1}}^{{+}^{k-1}{-}^{l+1}}
P_{{+}^{k-1}{-}^{l+1}}^{{+}^{k-1}{-}^{l+1}}
(t_{{+}{+}}^{{-}}\otimes P_{{+}^{k-1}{-}^{l}}^{{+}^{k-1}{-}^{l}})
(\mathbf{1}_{+}\otimes P_{{+}^{k}{-}^{l}}^{{+}^{k}{-}^{l}})g\\
&\quad=\frac{[k]}{[k+1]}
\,\tikz[baseline=-.6ex, scale=.1, yscale=-1]{
\begin{scope}[yshift=5cm]
\draw[-<-=.5] (-3,0) -- (-3,4);
\draw[-<-=.5] (2,1) -- (2,4);
\draw[->-=.5] (8,1) -- (8,4);
\draw[fill=white] (-2,0) rectangle (11,1);
\node at (2,4) [below]{${\scriptstyle k}$};
\node at (8,4) [below]{${\scriptstyle l}$};
\end{scope}
\draw (-3,2) -- (-3,5);
\draw[->-=.2, -<-=.8] (-3,-5) -- (-3,2);
\draw[-<-=.5] (-3,4) -- (0,4) -- (0,5);
\draw[->-=.5] (-3,2) -- (0,2) -- (0,1);
\draw[-<-=.5] (-3,-1) -- (0,-1) -- (0,0);
\draw[-<-=.5] (2,0) -- (2,5);
\draw[-<-=.5] (2,-5) -- (2,0);
\draw[->-=.5] (9,0) -- (9,5);
\draw[->-=.5] (9,-5) -- (9,0);
\draw[fill=white] (-1,0) rectangle (11,1);
\node at (2,-2) [right]{${\scriptstyle k-1}$};
\node at (9,-2) [right]{${\scriptstyle l}$};
\begin{scope}[yshift=-5cm, yscale=-1]
\draw[->-=.5] (0,1) -- (0,4);
\draw[-<-=.5] (9,1) -- (9,4);
\draw[fill=white] (-4,0) rectangle (11,1);
\node at (0,4) [above]{${\scriptstyle k-1}$};
\node at (9,4) [above]{${\scriptstyle l+1}$};
\end{scope}
}\,
+\frac{[l]}{[k+1][k+l+2]}
\,\tikz[baseline=-.6ex, scale=.1]{
\begin{scope}[yshift=5cm]
\draw[->-=.5] (0,1) -- (0,4);
\draw[-<-=.5] (6,1) -- (6,4);
\draw[fill=white] (-3,0) rectangle (9,1);
\node at (0,4) [above]{${\scriptstyle k-1}$};
\node at (6,4) [above]{${\scriptstyle l+1}$};
\end{scope}
\draw (-2,-2) -- (-2,5);
\draw (0,-6) -- (0,-3) -- (-2,-3) -- (-2,-6);
\draw[->-=.5] (0,0) -- (0,3) -- (-2,3);
\draw[->-=.5] (0,0) -- (0,-2) -- (-2,-2);
\draw[->-=.5] (2,-5) -- (2,0);
\draw[->-=.5] (3,0) -- (3,5);
\draw[-<-=.5] (6,-5) -- (6,0);
\draw[-<-=.5] (7,0) -- (7,5);
\draw[fill=white] (-1,0) rectangle (9,1);
\node at (6,-3) [right]{${\scriptstyle l-1}$};
\node at (7,3) [right]{${\scriptstyle l}$};
\begin{scope}[yshift=-5cm, yscale=-1]
\draw[-<-=.5] (-2,1) -- (-2,4);
\draw[-<-=.5] (2,1) -- (2,4);
\draw[->-=.5] (7,1) -- (7,4);
\draw[fill=white] (-1,0) rectangle (9,1);
\node at (-2,4) [below]{${\scriptstyle 1}$};
\node at (2,4) [below]{${\scriptstyle k}$};
\node at (7,4) [below]{${\scriptstyle l}$};
\end{scope}
}\,\\
&\quad=
\,\tikz[baseline=-.6ex, scale=.1, yscale=-.8]{
\begin{scope}[yshift=5cm]
\draw[-<-=.5] (-3,0) -- (-3,4);
\draw[-<-=.5] (1,1) -- (1,4);
\draw[->-=.5] (9,1) -- (9,4);
\draw[fill=white] (-2,0) rectangle (11,1);
\node at (1,4) [below]{${\scriptstyle k}$};
\node at (9,4) [below]{${\scriptstyle l}$};
\end{scope}
\draw (-3,2) -- (-3,5);
\draw[->-=.5] (-3,-5) -- (-3,2);
\draw[-<-=.5] (-3,3) -- (0,3) -- (0,5);
\draw[-<-=.5] (2,0) -- (2,5);
\draw (2,-5) -- (2,0);
\draw[->-=.5] (9,0) -- (9,5);
\draw (9,-5) -- (9,0);
\node at (2,-2) [right]{${\scriptstyle k-1}$};
\node at (9,-2) [right]{${\scriptstyle l}$};
\begin{scope}[yshift=-5cm, yscale=-1]
\draw[->-=.5] (0,1) -- (0,4);
\draw[-<-=.5] (8,1) -- (8,4);
\draw[fill=white] (-4,0) rectangle (11,1);
\node at (0,4) [above]{${\scriptstyle k-1}$};
\node at (8,4) [above]{${\scriptstyle l+1}$};
\end{scope}
}\,
=P_{{+}^{k-1}{-}^{l+1}}^{{+}^{k-1}{-}^{l+1}}
(t_{{+}{+}}^{{-}}\otimes P_{{+}^{k-1}{-}^{l}}^{{+}^{k-1}{-}^{l}})
(\mathbf{1}_{+}\otimes P_{{+}^{k}{-}^{l}}^{{+}^{k}{-}^{l}}).
\end{align*}

By Lemma~\ref{lem7} and Lemma~\ref{lemOY},
\begin{align*}
&P_{{+}^{k}{-}^{l-1}}^{{+}^{k}{-}^{l-1}}
(d_{{+}{-}}\otimes P_{{+}^{k}{-}^{l-1}}^{{+}^{k}{-}^{l-1}})
(\mathbf{1}_{+}\otimes P_{{+}^{k}{-}^{l}}^{{+}^{k}{-}^{l}})g\\
&\quad=\frac{[k]}{[k+1]}
\,\tikz[baseline=-.6ex, scale=.1]{
\begin{scope}[yshift=5cm]
\draw[->-=.5] (2,1) -- (2,4);
\draw[-<-=.5] (7,1) -- (7,4);
\draw[fill=white] (0,0) rectangle (9,1);
\node at (2,4) [above]{${\scriptstyle k}$};
\node at (7,4) [above]{${\scriptstyle l-1}$};
\end{scope}
\draw[->-=.5] (-2,-2) -- (-2,2);
\draw (-2,-5) -- (-2,-2);
\draw (0,0) -- (0,2) -- (-2,2);
\draw (0,0) -- (0,-2) -- (-2,-2);
\draw[->-=.8] (0,-5) -- (0,-4) -- (-2,-4);
\draw[->-=.5] (3,-5) -- (3,0);
\draw[->-=.5] (2,0) -- (2,5);
\draw[-<-=.5] (6,-5) -- (6,0);
\draw[-<-=.5] (7,0) -- (7,5);
\draw[fill=white] (-1,0) rectangle (9,1);
\begin{scope}[yshift=-5cm, yscale=-1]
\draw[-<-=.5] (-2,0) -- (-2,4);
\draw[-<-=.5] (2,1) -- (2,4);
\draw[->-=.5] (6,1) -- (6,4);
\draw[fill=white] (-1,0) rectangle (9,1);
\node at (2,4) [below]{${\scriptstyle k}$};
\node at (6,4) [below]{${\scriptstyle l}$};
\end{scope}
}\,
+\frac{[l]}{[k+1][k+l+2]}
\,\tikz[baseline=-.6ex, scale=.1]{
\begin{scope}[yshift=5cm]
\draw[->-=.5] (2,1) -- (2,4);
\draw[-<-=.5] (7,1) -- (7,4);
\draw[fill=white] (0,0) rectangle (9,1);
\node at (2,4) [above]{${\scriptstyle k}$};
\node at (7,4) [above]{${\scriptstyle l-1}$};
\end{scope}
\draw[->-=.5] (-2,-2) -- (-2,2);
\draw (-2,-5) -- (-2,-4);
\draw (0,0) -- (0,2) -- (-2,2);
\draw (0,0) -- (0,-2) -- (-2,-2);
\draw (0,-5) -- (0,-4) -- (-2,-4);
\draw[->-=.5] (2,-5) -- (2,0);
\draw[->-=.5] (2,0) -- (2,5);
\draw[-<-=.5] (7,-5) -- (7,0);
\draw[-<-=.5] (7,0) -- (7,5);
\draw[fill=white] (-1,0) rectangle (9,1);
\node at (-2,0) [left]{${\scriptstyle 1}$};
\begin{scope}[yshift=-5cm, yscale=-1]
\draw[-<-=.5] (-2,0) -- (-2,4);
\draw[-<-=.5] (2,1) -- (2,4);
\draw[->-=.5] (6,1) -- (6,4);
\draw[fill=white] (-1,0) rectangle (9,1);
\node at (-2,4) [below]{${\scriptstyle 1}$};
\node at (2,4) [below]{${\scriptstyle k}$};
\node at (6,4) [below]{${\scriptstyle l}$};
\end{scope}
}\, \\
&\quad=\left(1-\frac{[l+1]}{[k+1][k+l+1]}\right)
\,\tikz[baseline=-.6ex, scale=.1]{
\begin{scope}[yshift=3cm]
\draw[->-=.5] (0,1) -- (0,4);
\draw[-<-=.5] (6,1) -- (6,4);
\draw[fill=white] (-3,0) rectangle (8,1);
\node at (0,4) [above]{${\scriptstyle k}$};
\node at (6,4) [above]{${\scriptstyle l-1}$};
\end{scope}
\draw[->-=.5] (-2,-2) -- (-2,3);
\draw (-2,-5) -- (-2,-2);
\draw[-<-=.8] (1,-5) -- (1,-2) -- (-2,-2);
\draw[->-=.8] (0,-5) -- (0,-4) -- (-2,-4);
\draw[->-=.5] (3,-5) -- (3,3);
\draw[-<-=.5] (6,-5) -- (6,3);
\begin{scope}[yshift=-5cm, yscale=-1]
\draw[-<-=.5] (-2,0) -- (-2,4);
\draw[-<-=.5] (2,1) -- (2,4);
\draw[->-=.5] (5,1) -- (5,4);
\draw[fill=white] (-1,0) rectangle (8,1);
\node at (2,4) [below]{${\scriptstyle k}$};
\node at (5,4) [below]{${\scriptstyle l}$};
\end{scope}
}\,
+\frac{[l]}{[k+1][k+l+2]}\times\frac{[l+1][l+k+2]}{[l][l+k+1]}
\tikz[baseline=-.6ex, scale=.1]{
\begin{scope}[yshift=3cm]
\draw[->-=.5] (2,1) -- (2,4);
\draw[-<-=.5] (6,1) -- (6,4);
\draw[fill=white] (0,0) rectangle (8,1);
\node at (2,4) [above]{${\scriptstyle k}$};
\node at (6,4) [above]{${\scriptstyle l-1}$};
\end{scope}
\draw (-2,-3) -- (-2,-2);
\draw (0,-3) -- (0,-2) -- (-2,-2);
\draw[->-=.5] (2,-3) -- (2,3);
\draw[-<-=.5] (6,-3) -- (6,3);
\begin{scope}[yshift=-3cm, yscale=-1]
\draw[-<-=.5] (-2,0) -- (-2,4);
\draw[-<-=.5] (2,1) -- (2,4);
\draw[->-=.5] (5,1) -- (5,4);
\draw[fill=white] (-1,0) rectangle (8,1);
\node at (-2,4) [below]{${\scriptstyle 1}$};
\node at (2,4) [below]{${\scriptstyle k}$};
\node at (5,4) [below]{${\scriptstyle l}$};
\end{scope}
}\,\\
&\quad=
\,\tikz[baseline=-.6ex, scale=.1]{
\begin{scope}
\draw[->-=.5] (2,1) -- (2,4);
\draw[-<-=.5] (6,1) -- (6,4);
\draw[fill=white] (0,0) rectangle (8,1);
\node at (2,4) [above]{${\scriptstyle k}$};
\node at (6,4) [above]{${\scriptstyle l-1}$};
\end{scope}
\draw (-2,-3) -- (-2,-2);
\draw (0,-3) -- (0,-2) -- (-2,-2);
\draw[->-=.5] (2,-3) -- (2,0);
\draw[-<-=.5] (6,-3) -- (6,0);
\begin{scope}[yshift=-3cm, yscale=-1]
\draw[-<-=.5] (-2,0) -- (-2,4);
\draw[-<-=.5] (2,1) -- (2,4);
\draw[->-=.5] (5,1) -- (5,4);
\draw[fill=white] (-1,0) rectangle (8,1);
\node at (-2,4) [below]{${\scriptstyle 1}$};
\node at (2,4) [below]{${\scriptstyle k}$};
\node at (5,4) [below]{${\scriptstyle l}$};
\end{scope}
}\,
=(d_{{+}{-}}\otimes P_{{+}^{k}{-}^{l-1}}^{{+}^{k}{-}^{l-1}})
(\mathbf{1}_{+}\otimes P_{{+}^{k}{-}^{l}}^{{+}^{k}{-}^{l}}).
\end{align*}
Let us define $\iota_1\colon P_{(k-1,l+1)}\to ({+}^{k+1}{-}^{l},g)$ and $\iota_2\colon P_{(k,l-1)}\to ({+}^{k+1}{-}^{l},g)$ by 
$(\mathbf{1}_{+}\otimes P_{{+}^{k}{-}^{l}}^{{+}^{k}{-}^{l}})
(t_{{-}}^{{+}{+}}\otimes P_{{+}^{k-1}{-}^{l}}^{{+}^{k-1}{-}^{l}})
P_{{+}^{k-1}{-}^{l+1}}^{{+}^{k-1}{-}^{l+1}}$
and 
$(\mathbf{1}_{+}\otimes P_{{+}^{k}{-}^{l}}^{{+}^{k}{-}^{l}})
(b^{{+}{-}}\otimes P_{{+}^{k}{-}^{l-1}}^{{+}^{k}{-}^{l-1}})$, 
respectively.
Because these webs are obtained by turning the webs in $p_1$ and $p_2$ upside down,
we can confirm these maps are morphisms in $\Kar$ by the same calculation in the above.
The rest of the proof is to confirm that $p_1$, $p_2$, $\iota_1$, and $\iota_2$ satisfy the condition in Lemma~\ref{projincl}.
It is easy to see that $p_i\circ\iota_j=0$ if $i\neq j$.
$p_1\circ \iota_1=P_{{+}^{k-1}{-}^{l+1}}^{{+}^{k-1}{-}^{l+1}}$ and $p_2\circ \iota_2=P_{{+}^{k}{-}^{l-1}}^{{+}^{k}{-}^{l-1}}$ are derived from (\ref{eqprop5}) and Lemma~\ref{lemOY}, respectively.
By Proposition~\ref{prop5},
\begin{align*}
&\iota_1\circ p_1+\iota_2\circ p_2
=\frac{[k]}{[k+1]}
\,\tikz[baseline=-.6ex, scale=.1]{
\begin{scope}[yshift=4cm]
\draw[->-=.5] (-2,0) -- (-2,3);
\draw[->-=.5] (2,1) -- (2,3);
\draw[-<-=.5] (7,1) -- (7,3);
\draw[fill=white] (-1,0) rectangle (9,1);
\node at (-2,3) [above]{${\scriptstyle 1}$};
\node at (2,3) [above]{${\scriptstyle k}$};
\node at (7,3) [above]{${\scriptstyle l}$};
\end{scope}
\draw (-2,-4) -- (-2,4);
\draw[->-=.5] (1,-4) -- (1,0);
\draw[->-=.5] (1,1) -- (1,4);
\draw[-<-=.5] (8,-4) -- (8,0);
\draw[-<-=.5] (8,1) -- (8,4);
\draw[->-=.5] (-2,3) -- (0,3) -- (0,4);
\draw[-<-=.5] (-2,-3) -- (0,-3) -- (0,-4);
\draw[fill=white] (-3,0) rectangle (9,1);
\node at (.5,-2) [right]{${\scriptstyle k-1}$};
\node at (8,-2) [right]{${\scriptstyle l}$};
\begin{scope}[yshift=-4cm, yscale=-1]
\draw[-<-=.5] (-2,0) -- (-2,3);
\draw[-<-=.5] (2,1) -- (2,3);
\draw[->-=.5] (7,1) -- (7,3);
\draw[fill=white] (-1,0) rectangle (9,1);
\node at (-2,3) [below]{${\scriptstyle 1}$};
\node at (2,3) [below]{${\scriptstyle k}$};
\node at (7,3) [below]{${\scriptstyle l}$};
\end{scope}
}\,
+\frac{[l][k+l+1]}{[l+1][k+l+2]}
\,\tikz[baseline=-.6ex, scale=.1]{
\begin{scope}[yshift=3cm]
\draw[->-=.5] (-2,0) -- (-2,4);
\draw[->-=.5] (2,1) -- (2,4);
\draw[-<-=.5] (6,1) -- (6,4);
\draw[fill=white] (-1,0) rectangle (8,1);
\node at (-2,4) [above]{${\scriptstyle 1}$};
\node at (2,4) [above]{${\scriptstyle k}$};
\node at (6,4) [above]{${\scriptstyle l}$};
\end{scope}
\draw[->-=.5] (2,-3) -- (2,3);
\draw[-<-=.5] (6,-3) -- (6,3);
\draw (-2,3) to[out=south, in=west] (-1,1) to[out=east, in=south] (0,3);
\draw (-2,-3) to[out=north, in=west] (-1,-1) to[out=east, in=north](0,-3);
\node at (2,0) [right]{${\scriptstyle k}$};
\node at (6,0) [right]{${\scriptstyle l-1}$};
\begin{scope}[yshift=-3cm, yscale=-1]
\draw[-<-=.5] (-2,0) -- (-2,4);
\draw[-<-=.5] (2,1) -- (2,4);
\draw[->-=.5] (6,1) -- (6,4);
\draw[fill=white] (-1,0) rectangle (8,1);
\node at (-2,4) [below]{${\scriptstyle 1}$};
\node at (2,4) [below]{${\scriptstyle k}$};
\node at (6,4) [below]{${\scriptstyle l}$};
\end{scope}
}\,\\
&\quad=\frac{[k]}{[k+1]}
\,\tikz[baseline=-.6ex, scale=.1]{
\begin{scope}[yshift=4cm]
\draw[->-=.5] (-2,0) -- (-2,3);
\draw[->-=.5] (1,1) -- (1,3);
\draw[-<-=.5] (7,1) -- (7,3);
\draw[fill=white] (-1,0) rectangle (9,1);
\node at (-2,3) [above]{${\scriptstyle 1}$};
\node at (1,3) [above]{${\scriptstyle k}$};
\node at (7,3) [above]{${\scriptstyle l}$};
\end{scope}
\draw (-2,-4) -- (-2,4);
\draw[-<-=.5] (2,-4) -- (2,0);
\draw[-<-=.5] (2,1) -- (2,4);
\draw[->-=.5] (6,-4) -- (6,0);
\draw[->-=.5] (6,1) -- (6,4);
\draw[->-=.5] (-2,3) -- (0,3) -- (0,4);
\draw[-<-=.5] (-2,-3) -- (0,-3) -- (0,-4);
\draw[fill=white] (-3,0) rectangle (9,1);
\node at (2,-2) [right]{${\scriptstyle l}$};
\node at (6,-2) [right]{${\scriptstyle k-1}$};
\begin{scope}[yshift=-4cm, yscale=-1]
\draw[-<-=.5] (-2,0) -- (-2,3);
\draw[-<-=.5] (1,1) -- (1,3);
\draw[->-=.5] (7,1) -- (7,3);
\draw[fill=white] (-1,0) rectangle (9,1);
\node at (-2,3) [below]{${\scriptstyle 1}$};
\node at (1,3) [below]{${\scriptstyle k}$};
\node at (7,3) [below]{${\scriptstyle l}$};
\end{scope}
}\,
+\frac{[l][k+l+1]}{[l+1][k+l+2]}
\,\tikz[baseline=-.6ex, scale=.1]{
\begin{scope}[yshift=3cm]
\draw[->-=.5] (-2,0) -- (-2,4);
\draw[->-=.5] (2,1) -- (2,4);
\draw[-<-=.5] (6,1) -- (6,4);
\draw[fill=white] (-1,0) rectangle (8,1);
\node at (-2,4) [above]{${\scriptstyle 1}$};
\node at (2,4) [above]{${\scriptstyle k}$};
\node at (6,4) [above]{${\scriptstyle l}$};
\end{scope}
\draw[->-=.5] (2,-3) -- (2,3);
\draw[-<-=.5] (6,-3) -- (6,3);
\draw (-2,3) to[out=south, in=west] (-1,1) to[out=east, in=south] (0,3);
\draw (-2,-3) to[out=north, in=west] (-1,-1) to[out=east, in=north](0,-3);
\node at (2,0) [right]{${\scriptstyle k}$};
\node at (6,0) [right]{${\scriptstyle l-1}$};
\begin{scope}[yshift=-3cm, yscale=-1]
\draw[-<-=.5] (-2,0) -- (-2,4);
\draw[-<-=.5] (2,1) -- (2,4);
\draw[->-=.5] (6,1) -- (6,4);
\draw[fill=white] (-1,0) rectangle (8,1);
\node at (-2,4) [below]{${\scriptstyle 1}$};
\node at (2,4) [below]{${\scriptstyle k}$};
\node at (6,4) [below]{${\scriptstyle l}$};
\end{scope}
}\,\\
&\quad=\frac{[k]}{[k+1]}\left(
\,\tikz[baseline=-.6ex, scale=.1]{
\begin{scope}[yshift=4cm]
\draw[->-=.5] (-2,0) -- (-2,3);
\draw[->-=.5] (1,1) -- (1,3);
\draw[-<-=.5] (7,1) -- (7,3);
\draw[fill=white] (-1,0) rectangle (9,1);
\node at (-2,3) [above]{${\scriptstyle 1}$};
\node at (1,3) [above]{${\scriptstyle k}$};
\node at (7,3) [above]{${\scriptstyle l}$};
\end{scope}
\draw (-2,-4) -- (-2,4);
\draw (2,-4) -- (2,1);
\draw[-<-=.5] (2,1) -- (2,4);
\draw (6,-4) -- (6,1);
\draw[->-=.5] (6,1) -- (6,4);
\draw[->-=.5] (-2,3) -- (0,3) -- (0,4);
\draw[-<-=.5] (-2,-3) -- (0,-3) -- (0,-4);
\node at (2,0) [right]{${\scriptstyle l}$};
\node at (6,0) [right]{${\scriptstyle k-1}$};
\begin{scope}[yshift=-4cm, yscale=-1]
\draw[-<-=.5] (-2,0) -- (-2,3);
\draw[-<-=.5] (1,1) -- (1,3);
\draw[->-=.5] (7,1) -- (7,3);
\draw[fill=white] (-1,0) rectangle (9,1);
\node at (-2,3) [below]{${\scriptstyle 1}$};
\node at (1,3) [below]{${\scriptstyle k}$};
\node at (7,3) [below]{${\scriptstyle l}$};
\end{scope}
}\,
-\frac{[l]}{[l+1]}
\,\tikz[baseline=-.6ex, scale=.1]{
\begin{scope}[yshift=4cm]
\draw[->-=.5] (-2,0) -- (-2,3);
\draw[->-=.5] (1,1) -- (1,3);
\draw[-<-=.5] (7,1) -- (7,3);
\draw[fill=white] (-1,0) rectangle (9,1);
\node at (-2,3) [above]{${\scriptstyle 1}$};
\node at (1,3) [above]{${\scriptstyle k}$};
\node at (7,3) [above]{${\scriptstyle l}$};
\end{scope}
\draw (-2,-4) -- (-2,4);
\draw (2,-4) -- (2,1);
\draw[-<-=.5] (2,1) -- (2,4);
\draw (8,-4) -- (8,1);
\draw[->-=.5] (8,1) -- (8,4);
\draw[->-=.5] (-2,3) -- (0,3) -- (0,4);
\draw[-<-=.2] (-2,1) -- (1,1) -- (1,4);
\draw[-<-=.5] (-2,-3) -- (0,-3) -- (0,-4);
\draw[->-=.2] (-2,-1) -- (1,-1) -- (1,-4);
\node at (1.5,0) [right]{${\scriptstyle l-1}$};
\node at (8,0) [right]{${\scriptstyle k-1}$};
\begin{scope}[yshift=-4cm, yscale=-1]
\draw[-<-=.5] (-2,0) -- (-2,3);
\draw[-<-=.5] (1,1) -- (1,3);
\draw[->-=.5] (7,1) -- (7,3);
\draw[fill=white] (-1,0) rectangle (9,1);
\node at (-2,3) [below]{${\scriptstyle 1}$};
\node at (1,3) [below]{${\scriptstyle k}$};
\node at (7,3) [below]{${\scriptstyle l}$};
\end{scope}
}\,
\right)
+\frac{[l][k+l+1]}{[l+1][k+l+2]}
\,\tikz[baseline=-.6ex, scale=.1]{
\begin{scope}[yshift=3cm]
\draw[->-=.5] (-2,0) -- (-2,4);
\draw[->-=.5] (2,1) -- (2,4);
\draw[-<-=.5] (6,1) -- (6,4);
\draw[fill=white] (-1,0) rectangle (8,1);
\node at (-2,4) [above]{${\scriptstyle 1}$};
\node at (2,4) [above]{${\scriptstyle k}$};
\node at (6,4) [above]{${\scriptstyle l}$};
\end{scope}
\draw[->-=.5] (2,-3) -- (2,3);
\draw[-<-=.5] (6,-3) -- (6,3);
\draw (-2,3) to[out=south, in=west] (-1,1) to[out=east, in=south] (0,3);
\draw (-2,-3) to[out=north, in=west] (-1,-1) to[out=east, in=north](0,-3);
\node at (2,0) [right]{${\scriptstyle k}$};
\node at (6,0) [right]{${\scriptstyle l-1}$};
\begin{scope}[yshift=-3cm, yscale=-1]
\draw[-<-=.5] (-2,0) -- (-2,4);
\draw[-<-=.5] (2,1) -- (2,4);
\draw[->-=.5] (6,1) -- (6,4);
\draw[fill=white] (-1,0) rectangle (8,1);
\node at (-2,4) [below]{${\scriptstyle 1}$};
\node at (2,4) [below]{${\scriptstyle k}$};
\node at (6,4) [below]{${\scriptstyle l}$};
\end{scope}
}\,\\
&\quad=\frac{[k]}{[k+1]}
\,\tikz[baseline=-.6ex, scale=.1]{
\begin{scope}[yshift=4cm]
\draw[->-=.5] (-2,0) -- (-2,3);
\draw[->-=.5] (1,1) -- (1,3);
\draw[-<-=.5] (7,1) -- (7,3);
\draw[fill=white] (-1,0) rectangle (9,1);
\node at (-2,3) [above]{${\scriptstyle 1}$};
\node at (1,3) [above]{${\scriptstyle k}$};
\node at (7,3) [above]{${\scriptstyle l}$};
\end{scope}
\draw (-2,-4) -- (-2,4);
\draw (2,-4) -- (2,1);
\draw[-<-=.5] (2,1) -- (2,4);
\draw (6,-4) -- (6,1);
\draw[->-=.5] (6,1) -- (6,4);
\draw[->-=.5] (-2,3) -- (0,3) -- (0,4);
\draw[-<-=.5] (-2,-3) -- (0,-3) -- (0,-4);
\node at (2,0) [right]{${\scriptstyle l}$};
\node at (6,0) [right]{${\scriptstyle k-1}$};
\begin{scope}[yshift=-4cm, yscale=-1]
\draw[-<-=.5] (-2,0) -- (-2,3);
\draw[-<-=.5] (1,1) -- (1,3);
\draw[->-=.5] (7,1) -- (7,3);
\draw[fill=white] (-1,0) rectangle (9,1);
\node at (-2,3) [below]{${\scriptstyle 1}$};
\node at (1,3) [below]{${\scriptstyle k}$};
\node at (7,3) [below]{${\scriptstyle l}$};
\end{scope}
}\,
-\frac{[l]}{[l+1]}\left(\frac{[k+l+1]}{[k+l+2]}-\frac{[k]}{[k+1]}\right)
\,\tikz[baseline=-.6ex, scale=.1]{
\begin{scope}[yshift=3cm]
\draw[->-=.5] (-2,0) -- (-2,4);
\draw[->-=.5] (2,1) -- (2,4);
\draw[-<-=.5] (6,1) -- (6,4);
\draw[fill=white] (-1,0) rectangle (8,1);
\node at (-2,4) [above]{${\scriptstyle 1}$};
\node at (2,4) [above]{${\scriptstyle k}$};
\node at (6,4) [above]{${\scriptstyle l}$};
\end{scope}
\draw[->-=.5] (2,-3) -- (2,3);
\draw[-<-=.5] (6,-3) -- (6,3);
\draw (-2,3) to[out=south, in=west] (-1,1) to[out=east, in=south] (0,3);
\draw (-2,-3) to[out=north, in=west] (-1,-1) to[out=east, in=north](0,-3);
\node at (2,0) [right]{${\scriptstyle k}$};
\node at (6,0) [right]{${\scriptstyle l-1}$};
\begin{scope}[yshift=-3cm, yscale=-1]
\draw[-<-=.5] (-2,0) -- (-2,4);
\draw[-<-=.5] (2,1) -- (2,4);
\draw[->-=.5] (6,1) -- (6,4);
\draw[fill=white] (-1,0) rectangle (8,1);
\node at (-2,4) [below]{${\scriptstyle 1}$};
\node at (2,4) [below]{${\scriptstyle k}$};
\node at (6,4) [below]{${\scriptstyle l}$};
\end{scope}
}\,\\
&\quad=\frac{[k]}{[k+1]}
\,\tikz[baseline=-.6ex, scale=.1]{
\begin{scope}[yshift=4cm]
\draw[->-=.5] (-2,0) -- (-2,3);
\draw[->-=.5] (1,1) -- (1,3);
\draw[-<-=.5] (7,1) -- (7,3);
\draw[fill=white] (-1,0) rectangle (9,1);
\node at (-2,3) [above]{${\scriptstyle 1}$};
\node at (1,3) [above]{${\scriptstyle k}$};
\node at (7,3) [above]{${\scriptstyle l}$};
\end{scope}
\draw (-2,-4) -- (-2,4);
\draw (2,-4) -- (2,1);
\draw[-<-=.5] (2,1) -- (2,4);
\draw (6,-4) -- (6,1);
\draw[->-=.5] (6,1) -- (6,4);
\draw[->-=.5] (-2,3) -- (0,3) -- (0,4);
\draw[-<-=.5] (-2,-3) -- (0,-3) -- (0,-4);
\node at (2,0) [right]{${\scriptstyle l}$};
\node at (6,0) [right]{${\scriptstyle k-1}$};
\begin{scope}[yshift=-4cm, yscale=-1]
\draw[-<-=.5] (-2,0) -- (-2,3);
\draw[-<-=.5] (1,1) -- (1,3);
\draw[->-=.5] (7,1) -- (7,3);
\draw[fill=white] (-1,0) rectangle (9,1);
\node at (-2,3) [below]{${\scriptstyle 1}$};
\node at (1,3) [below]{${\scriptstyle k}$};
\node at (7,3) [below]{${\scriptstyle l}$};
\end{scope}
}\,
-\frac{[l]}{[k+1][k+l+2]}
\,\tikz[baseline=-.6ex, scale=.1]{
\begin{scope}[yshift=3cm]
\draw[->-=.5] (-2,0) -- (-2,4);
\draw[->-=.5] (2,1) -- (2,4);
\draw[-<-=.5] (6,1) -- (6,4);
\draw[fill=white] (-1,0) rectangle (8,1);
\node at (-2,4) [above]{${\scriptstyle 1}$};
\node at (2,4) [above]{${\scriptstyle k}$};
\node at (6,4) [above]{${\scriptstyle l}$};
\end{scope}
\draw[->-=.5] (2,-3) -- (2,3);
\draw[-<-=.5] (6,-3) -- (6,3);
\draw (-2,3) to[out=south, in=west] (-1,1) to[out=east, in=south] (0,3);
\draw (-2,-3) to[out=north, in=west] (-1,-1) to[out=east, in=north](0,-3);
\node at (2,0) [right]{${\scriptstyle k}$};
\node at (6,0) [right]{${\scriptstyle l-1}$};
\begin{scope}[yshift=-3cm, yscale=-1]
\draw[-<-=.5] (-2,0) -- (-2,4);
\draw[-<-=.5] (2,1) -- (2,4);
\draw[->-=.5] (6,1) -- (6,4);
\draw[fill=white] (-1,0) rectangle (8,1);
\node at (-2,4) [below]{${\scriptstyle 1}$};
\node at (2,4) [below]{${\scriptstyle k}$};
\node at (6,4) [below]{${\scriptstyle l}$};
\end{scope}
}\, .
\end{align*}
The last equation uses a formula $[a][b]=\sum_{i=1}^{a}[a+b-(2i-1)]$ where $a$ and $b$ are integers.
Thus, the condition in Lemma~\ref{projincl} is satisfied and we obtained the isomorphism $({+}^{k+1}{-}^l,g)\cong P_{(k-1,l+1)}\oplus P_{(k,l-1)}$ in $\Kar$.
In terms of $K_0(\Kar)$, the equation (\ref{eqCK3}) is interpreted as
\begin{align*}
X\langle P_{(k,l)}\rangle-\langle P_{(k+1,l)}\rangle=\langle({+}^{k+1}{-}^l,g)\rangle=\langle P_{(k-1,l+1)}\rangle+\langle P_{(k,l-1)}\rangle.
\end{align*}

\end{proof}

\bibliographystyle{amsalpha}
\bibliography{2Chebyshev}
\end{document}